\documentclass[11pt,twoside,reqno]{amsart}

\usepackage{a4wide}
\usepackage[T1]{fontenc}
\usepackage[utf8]{inputenc}
\usepackage{amsmath}
\usepackage{amsfonts}
\usepackage{amssymb}
\usepackage{mathtools}
\usepackage{amsthm}
\usepackage{xcolor}
\usepackage{graphicx}
\usepackage[hidelinks]{hyperref}
\usepackage{standalone}
\usepackage{esint}
\usepackage{comment}
\usepackage[left=2.3cm,top=3cm,right=2.3cm]{geometry} 
\usepackage{bm}
\usepackage{enumerate}
\usepackage{cite}

\newcommand{\R}{\mathbb{R}}

\newcommand{\C}{\mathbb{C}}

\newcommand{\N}{\mathbb{N}}

\newcommand{\bH}{\mathbb{H}}

\renewcommand{\S}{\mathbb{S}}

\newcommand{\cM}{\mathcal{M}}

\newcommand{\eps}{\varepsilon}

\newcommand{\dist}{\operatorname{dist}}

\renewcommand{\emptyset}{\varnothing}
\renewcommand{\epsilon}{\varepsilon}
\renewcommand{\rho}{\varrho}
\renewcommand{\phi}{\varphi}

\newcommand{\bHS}{\mathrm{HS}}
\newcommand{\supp}{\mathrm{supp}}

\newcommand{\id}{\operatorname{id}}

\DeclareMathOperator{\Vol}{Vol}
\DeclareMathOperator{\Tr}{Tr}
\DeclareMathOperator{\vol}{Vol}
\DeclareMathOperator{\InjRad}{InjRad}

\DeclareMathOperator{\PSL}{PSL}

\newcommand{\norm}[1]{\left\lVert#1\right\rVert}

\numberwithin{equation}{section}

\theoremstyle{plain}
\newtheorem{thm}{Theorem}[section]
\newtheorem{theorem}[thm]{Theorem}

\newtheorem{lemma}[thm]{Lemma}

\theoremstyle{definition}

\theoremstyle{remark}
\newtheorem{remark}[thm]{Remark}

\title{Quantum Mixing for Schr\"odinger eigenfunctions in Benjamini-Schramm limit}

\author{Kai Hippi}
\address{Aalto University, Espoo, Finland}
\email{kai.hippi@aalto.fi}

\author{F\'{e}lix Lequen}
\address{Université Sorbonne Paris Nord, Villetaneuse, France}
\email{lequen@math.univ-paris13.fr}

\author{S{\o}ren Mikkelsen}
\address{University of Helsinki, Helsinki, Finland}
\email{soren.mikkelsen@helsinki.fi}

\author{Tuomas Sahlsten}
\address{University of Helsinki, Helsinki, Finland}
\email{tuomas.sahlsten@helsinki.fi}

\author{Henrik Uebersch\"{a}r}
\address{Sorbonne Universit\'e, Paris, France}
\email{henrik.ueberschar@imj-prg.fr}

\thanks{K.H., S.M., and T.S. acknowledge support from the Research Council of Finland's Academy Research Fellowship \emph{``Quantum chaos of large and many body systems''}, grant Nos. 347365, 353738. F.L. also received support from this grant during 2023, from
Tuomas Orponen’s grant from the Research Council of Finland via the project Approximate Incidence Geometry, grant no. 355453 and the joint ANR-SNF project \emph{``Equidistribution in Number Theory''} (FNS no. 10.003.145 and ANR-24-CE93-0016). K.H. is also supported by the Vilho, Yrjö and Kalle Väisälä Foundation. AI tools were used in proof-reading. All original ideas are those of the authors, who independently checked the final manuscript.}

\begin{document}

\begin{abstract}
Let $\{X_n\}_{n\in\N}$ be a  sequence of compact hyperbolic surfaces which is uniformly discrete and Benjamini-Schramm converges to $\mathbb{H}$ and let $\{V_n\}_{n\in\N}$ be a sequence of potentials such that the $L^2$-norm of $V_n$ is $o(1)$ with respect to the volume of $X_n$. We prove quantum mixing for the eigenfunctions of $-\Delta_{X_n}+V_n$ in any sufficiently large spectral window $I$ for bounded observables which are polynomial mixing with respect to the geodesic flow. Examples of such sequences of potentials include point-cloud potentials just below the thermodynamic limit, Hartree potentials for dilute Bose gases in hyperbolic space, and sequences induced by bounded potentials in $L^p(\mathbb{H})$ for some $p>0$. This is the first result of this kind beyond the locally symmetric setting. The proof combines classical geodesic flow mixing on $T^1X_n$, replacing Nevo's ergodic theorem, with the Duhamel formula and recently developed geometric wave-kernel estimates by the first author.
\end{abstract}

\maketitle


\section{Introduction}

\subsection{Background}

Understanding eigenfunctions of quantum Hamiltonians is central in both physics and mathematics, as they determine the time evolution and statistical behaviour of quantum systems. A guiding principle in quantum chaos is that spectral properties of a quantum system reflect the dynamics of the associated classical flow. In particular, ergodic classical dynamics is expected to lead to eigenfunction delocalisation. This is formalised by the quantum ergodicity theorem of Shnirelman, Zelditch, and Colin de Verdi\`ere \cite{shnirelmanErgodicPropertiesEigenfunctions,zelditchUniformDistributionEigenfunctions1987,colindeverdiereErgodiciteFonctionsPropres1985}, which asserts that on manifolds with ergodic geodesic flow, eigenfunctions of the Laplace-Beltrami operator equidistribute in the high-energy limit outside a density-zero subsequence. This principle has motivated a large body of work on eigenfunction delocalisation and quantum chaos; see, e.g. \cite{Anantharaman2008,AnantharamanNonnenmacher2007,Lindenstrauss2006,Soundararajan2010,DyatlovZahl2016,DyatlovJin2021} and the surveys \cite{Anantharaman2022,Dyatlov2022}.

A related regime arises when the underlying space itself grows. For sequences of regular graphs converging locally to trees, quantum ergodicity was established in \cite{anantharamanQuantumErgodicityLarge2015,brooks2015quantumergodicityaveragingoperators}, building on earlier delocalisation results \cite{brooksNonlocalizationEigenfunctionsLarge2013}. These results were subsequently extended to non-regular graphs, Cayley graphs, and Schr\"odinger operators, including Anderson-type models under assumptions on the limiting spectral measure \cite{AnantharamanSabriAnnals,anantharamanQuantumErgodicityAnderson2017,AnantharamanSabri2019,MageeThomasZhao2023,bordenave2026quantummixinglargeschreier}. Related developments in random matrix theory and random graphs include local laws and eigenvector delocalisation \cite{erdosLocalSemicircleLaw2009,erdosSpectralStatisticsErdosRenyi2013,bauerschmidtLocalSemicircleLaw2017,bauerschmidtLocalKestenMcKay2019}.

In the continuous setting, the natural analogue is given by sequences of hyperbolic surfaces whose genus tends to infinity and which converge locally to the hyperbolic plane in the Benjamini-Schramm sense. Quantum unique ergodicity for arithmetic families in this setting was considered by Nelson \cite{Nelson_2011} and by Nelson, Pitale, and Saha \cite{NelsonPitaleSaha} in the level aspect. Quantum ergodicity in the large-genus regime was established by Le Masson and the fourth author \cite{lemassonQuantumErgodicityBenjaminiSchramm2017}, who proved equidistribution in fixed spectral windows under natural geometric assumptions such as a uniform injectivity radius and spectral gap. Many related developments have followed; see, e.g. \cite{lemassonQuantumErgodicityEisenstein2024,abertEigenfunctionsRandomWaves2018,AnantharamanSaha,peterson2023quantumergodicitybruhattitsbuilding,brumley2020quantumergodicitycompactquotients,brumley2026quantumergodicitybenjaminischrammlimit,monkGeometrySpectrumTypical,monkBenjaminiSchrammConvergence2022,gilmoreShortGeodesicLoops2021,KimTao2026,Gavelli2026}. There has also been substantial recent progress in transferring ideas from tree-like graph models to random hyperbolic surfaces, particularly in the study of optimal spectral gaps \cite{AnantharamanMonk2023,AnantharamanMonk2024a,AnantharamanMonk2024b,AnantharamanMonk2025,HideMaceraThomas2025a,HideMaceraThomas2025b,MageePuderVanHandel2025,HideMagee,BordenaveCollins,Chen_2026,KimTao2026}.

Beyond quantum ergodicity, one can ask for quantum mixing, introduced by Zelditch \cite{Zel90,Zel96a} in the high-energy setting and adapted here to the large-volume limit. An analogue for Wigner matrices was proved by Cipolloni, Erd\H{o}s, and Yau \cite{CES21} in the form of eigenstate thermalisation. In the geometric setting, the first author proved quantum mixing for large hyperbolic surfaces, extending the earlier quantum ergodicity results \cite{lemassonQuantumErgodicityBenjaminiSchramm2017,lemassonQuantumErgodicityEisenstein2024}; more recently, Bordenave, Letrouit, and Sabri established quantum mixing for sequences of Schreier graphs \cite{bordenave2026quantummixinglargeschreier} using strong convergence \cite{BordenaveCollins}.

Most results in the continuous setting, however, concern the Laplacian, corresponding to free motion. Introducing a potential leads to the Schr\"odinger operator
$$
H=-\Delta+V,
$$
which destroys the symmetries underlying spherical analysis and many of the existing large-volume arguments. Our purpose is to show that quantum ergodicity and quantum mixing nevertheless persist for a broad class of deterministic potentials on large hyperbolic surfaces. More precisely, we consider Schr\"odinger operators
$$
H_{V_n}=-\Delta_{X_n}+V_n
$$
on sequences of compact hyperbolic surfaces $X_n$. The key condition is that the potentials are uniformly bounded and satisfy
$$
\norm{V_n}_{L^2(X_n)}^2=o(\Vol(X_n)), \qquad n\to\infty.
$$
Thus the perturbation is small only in an averaged sense over the growing surface. In particular, $V_n$ need not converge to zero locally, so this is not merely a weak-coupling regime. The condition allows order-one perturbations supported on a vanishing proportion of the surface and arises naturally in several rather different settings. It covers:
\begin{enumerate}
    \item sparse point-cloud potentials consisting of uniformly bounded scatterers centred at $o(g_n)$ well-separated points;
    \item effective Hartree potentials arising in the mean-field approximation of dilute Bose gases on hyperbolic surfaces;
    \item potentials induced by fixing $V\in L^p(\bH)\cap L^\infty(\bH)$ for some $p>0$ and restricting $V$ to fundamental domains before extending periodically.
\end{enumerate}
These examples are discussed in detail in Section~\ref{SEC:potential}.

\subsection{The main result}

We now describe the assumptions and state the main theorem. Let $\{X_n\}_{n\in\N}$ be a sequence of compact connected hyperbolic surfaces. We assume the following two geometric conditions as in \cite{lemassonQuantumErgodicityBenjaminiSchramm2017,Hippi,lemassonQuantumErgodicityEisenstein2024,abertEigenfunctionsRandomWaves2018,AnantharamanSaha,peterson2023quantumergodicitybruhattitsbuilding,brumley2020quantumergodicitycompactquotients,brumley2026quantumergodicitybenjaminischrammlimit,Gavelli2026},:
\begin{enumerate}[label={ABC}]
    \item[\textbf{(BSC)}:] Benjamini-Schramm convergence: $\{X_n\}_{n\in\N}$ converges to $\mathbb H$ in the Benjamini-Schramm topology. That is, for every $R>0$,
    \[
        \frac{|\{x\in X_n:\InjRad_{X_n}(x)\leq R\}|}{|X_n|}
        \longrightarrow 0,
        \qquad n\to\infty.
    \]

    \item[\textbf{(UND)}:] Uniform discreteness: the injectivity radii $\InjRad_{X_n}$ have a uniform positive lower bound.
\end{enumerate}

Now, we consider the Schr\"odinger operators on the surfaces $X_n$
\[
    H_{V_n}=-\Delta_{X_n}+V_n,
\]
where $-\Delta_{X_n}$ is the positive Laplace-Beltrami operator the potentials $V_n : X_n \to \mathbb R$ satisfy the following condition:
\begin{enumerate}[label={ABC}]
    \item[\textbf{(POT)}:] There exist $C_{\min},C_{\max}\in\R$ such that
    \[
        C_{\min}\leq V_n\leq C_{\max}
        \qquad\text{for all }n\in\N,
    \]
    and
    \[
        \norm{V_n}_{L^2(X_n)}^2=o(\Vol(X_n)).
    \]
\end{enumerate}

By Gauss-Bonnet, $\Vol(X_n)=4\pi(g_n-1)$, so the last condition is equivalent to $\norm{V_n}_{L^2(X_n)}^2=o(g_n)$. The examples above show that \textbf{(POT)} allows genuinely non-vanishing potentials and in particular goes beyond a weak-coupling regime. For instance, a point cloud may contain an increasing number of scatterers as long as their density tends to zero, while a potential obtained from a fixed $V\in L^p(\bH)\cap L^\infty(\bH)$ need not vanish locally in the limit. See Section \ref{SEC:potential} for more details on this.

The final assumption is a quantitative (classical) mixing assumption on the physical space observables with which we test quantum ergodicity and quantum mixing. Let $a_n$ be a sequence of $L^2$ integrable observables on $X_n$ with $\sup_n\|a_n\|_\infty<\infty$. We say that they satisfy the following uniform polynomial mixing estimate:
\begin{enumerate}[label={ABC}]
    \item[\textbf{(MIX)}:] There exist $C>0$ and $\kappa>1$ such that, for every $n\in\N$,
    \[
    \Big|
        \langle a_n\circ\phi_t,a_n\rangle
        -
        \Big(\frac{1}{\Vol(X_n)}\int_{X_n}a_n\Big)^2
    \Big|
    \leq
    C(t+1)^{-\kappa}\norm{a_n}_{L^2(X_n)}^2,
    \qquad t\geq0,
    \]
    where $\phi_t$ is the geodesic flow on $T^1X_n$ and $a_n(x,\theta)=a_n(x)$ for $(x,\theta)\in T^1X_n$.
\end{enumerate}

Thus, instead of imposing a uniform spectral gap on the surfaces, as in earlier approaches based on Nevo's ergodic theorem and the Kunze-Stein phenomenon, we assume directly the quantitative correlation decay required by the proof. By the work of Ratner \cite{Ratner}, a uniform spectral gap $\lambda_1(X_n)\geq c_0>0$ implies \textbf{(MIX)} for every uniformly bounded sequence $(a_n)$, in fact with exponential decay; see Theorem~\ref{Thm: Exponential mixing}. Our argument only requires polynomial decay with exponent $\kappa>1$. This formulation is analogous to the mixing hypothesis in the semiclassical quantum mixing results of Zelditch \cite{Zel96a} and is also closer in spirit to the recent work of Bordenave, Letrouit, and Sabri \cite{bordenave2026quantummixinglargeschreier} on Schreier graphs.

Under these conditions, we prove:

\begin{theorem}\label{thm:main}
    Let $ \{ X_{n} \}_{n\in\N} $ be a sequence of compact connected hyperbolic surfaces that satisfies {\normalfont\textbf{(BSC)}}, {\normalfont\textbf{(UND)}} and let $\{ V_{n} \}_{n\in\N} $ be a sequence of potentials that satisfies {\normalfont\textbf{(POT)}} with the two numbers $C_{\min}$ and $C_{\max}$. Consider for each $n\in\N$ the Schr\"odinger operator  $H_{V_n}^{(n)}= -\Delta_{X_n}+V_n$ with eigenvalues $\lambda_1^{(n)}\leq \lambda_2^{(n)}\leq \lambda_3^{(n)} \leq \lambda_4^{(n)} \cdots $ and eigenvectors $\{\psi_j^{(n)}\}_{j\in\N}$. For all $n,j\in\N$ set $\rho_{j}^{(n)} = \sqrt{\lambda_j^{(n)}-\frac{1}{4}}$. Fix $\frac{1}{4}< a' < a < b < b'$  such that $b-C_{\max}>a  -\min(0, C_{\min})$ and a sequence of spectral windows $I_n$ such that $[a,b] \subset I_n \subset [a',b']$. Then for any uniformly bounded sequence of measurable functions $\{a_n\}_{n\in\N}$ satisfying {\normalfont\textbf{(MIX)}} we have the following three properties:
    \begin{enumerate}
        \item[\emph{(1)}]
        \begin{equation*}
        \lim_{n\rightarrow \infty}\frac{1}{\mathcal{N}\big(H_{V_n}^{(n)},I_n\big)} \sum_{j:\lambda_j^{(n)}\in I_n} \Big| \langle a_n \psi_j^{(n)} , \psi_j^{(n)}\rangle_{L^2(X_n)} - \frac{1}{\Vol(X_n)} \int_{X_n} a_n(x)  \, dx \Big|^2 =0,
        \end{equation*}
        where $dx$ is the volume form on $X_n$.

        \item[\emph{(2)}] For every $\varepsilon>0$ there exists $\delta(\varepsilon)>0$ such that
        \begin{equation*}
        \limsup_{n\rightarrow \infty}\frac{1}{\mathcal{N}\big(H_{V_n}^{(n)},I_n\big)}  \sum_{\substack{j\neq k: \lambda_j^{(n)},\lambda_k^{(n)}\in I_n \\ |\rho_{j}^{(n)}-\rho_{k}^{(n)}| < \delta(\varepsilon)}} \Big| \langle a_n \psi_j^{(n)} , \psi_k^{(n)}\rangle_{L^2(X_n)}  \Big|^2 <\varepsilon.
        \end{equation*}

        \item[\emph{(3)}] For every $\varepsilon>0$ there exists $\delta(\varepsilon)>0$ such that for all $\tau\in\R$
        \begin{equation*}
        \limsup_{n\rightarrow \infty}\frac{1}{\mathcal{N}\big(H_{V_n}^{(n)},I_n\big)}  \sum_{\substack{j\neq k: \lambda_j^{(n)},\lambda_k^{(n)}\in I_n  \\ |\rho_{j}^{(n)}-\rho_{k}^{(n)}-\tau| < \delta(\varepsilon)}} \Big| \langle a_n \psi_j^{(n)} , \psi_k^{(n)}\rangle_{L^2(X_n)}  \Big|^2 <\varepsilon.
        \end{equation*}
    \end{enumerate}
    In all three cases $\mathcal{N}(H_{V_n}^{(n)},I_n)$ is the number of eigenvalues of $H_{V_n}^{(n)}$ in $I_n$ counted with multiplicity.
\end{theorem}
\begin{remark}
The three conclusions in Theorem~\ref{thm:main} describe spectrally averaged forms of eigenstate equidistribution, both on and off the diagonal. 
\begin{itemize}
    \item  Part~(1) is quantum ergodicity in the large-volume limit: for a density-one family of eigenfunctions, the expectation of the observable converges to its spatial average. \item 
    Parts~(2) and~(3) describe the corresponding off-diagonal behaviour. Part~(2) treats pairs of eigenfunctions whose eigenvalues become asymptotically close, while part~(3) is the quantum mixing statement, controlling off-diagonal matrix elements between eigenfunctions whose spectral separation converges to a prescribed frequency $\tau\in\mathbb R$. In particular, part~(2) is the case $\tau=0$ of part~(3). This is how quantum mixing was written in case of semiclassical limits \cite{Zel96a} as analogues of (1)-(2) characterise classical ergodicity of the geodesic flow, while properties (1)-(3) characterise weak mixing.
\end{itemize}

This diagonal/off-diagonal structure is reminiscent of the eigenstate thermalisation hypothesis \cite{deutschEigenstateThermalizationHypothesis2018} on the behaviour matrix elements of observables in thermodynamic limit of many-body systems near a given energy level. Our setting is different, since we consider a spectrally averaged one-particle large-volume limit in a given energy window rather than a many-body thermodynamic limit.
\end{remark}

The main novelty is that these conclusions persist for Schr\"odinger operators with non-trivial potentials. Previous large-volume quantum ergodicity and quantum mixing results on locally symmetric spaces \cite{lemassonQuantumErgodicityBenjaminiSchramm2017,abertEigenfunctionsRandomWaves2018,brumley2026quantumergodicitybenjaminischrammlimit,brumley2020quantumergodicitycompactquotients,peterson2023quantumergodicitybruhattitsbuilding} exploit, in different ways, the symmetries of the underlying Laplace or averaging operators. A general potential destroys these symmetries, so the corresponding spherical or convolution-based arguments are no longer available. 

Our approach instead uses finite-time wave propagation and the Duhamel formula to separate the relevant quantity into a free contribution and a perturbative contribution generated by the potential. The free part is reduced, using the wave-kernel estimates developed in \cite{Hippi}, to the classical correlation decay in \textbf{(MIX)}, while the perturbed part is controlled using the $L^2$-bound in \textbf{(POT)}, Benjamini-Schramm convergence, and geometric estimates for the corresponding integral kernels. This separation explains both the role of mixing and the form of the potential assumption: the condition $\norm{V_n}_{L^2(X_n)}^2=o(g_n)$ makes the perturbative contribution negligible compared with the volume, even though the potential may remain of order one locally.

The theorem applies, in particular, to a fixed spectral window $I_n=I$ for all $n$, which is the setting more commonly considered in \cite{lemassonQuantumErgodicityBenjaminiSchramm2017,lemassonQuantumErgodicityEisenstein2024,abertEigenfunctionsRandomWaves2018,peterson2023quantumergodicitybruhattitsbuilding,brumley2020quantumergodicitycompactquotients,brumley2026quantumergodicitybenjaminischrammlimit,Hippi,Anantharaman2022,bordenave2026quantummixinglargeschreier,AnantharamanSabri2019,AnantharamanSabriAnnals,brooks2015quantumergodicityaveragingoperators}. We also expect that the result should extend to suitable phase-space observables after restricting to a fixed finite-energy band, along the lines of the extension from multiplication operators to more general local observables in \cite{abertEigenfunctionsRandomWaves2018,AnantharamanSaha,Gavelli2026}, provided that an appropriate phase-space analogue of \textbf{(MIX)} holds for their symbols. We do not pursue this extension here.

\subsection{Examples of surfaces and a probabilistic version}

The assumptions on the surfaces are satisfied by natural deterministic and
random families. Conditions \textbf{(BSC)}, \textbf{(UND)}, and
\textbf{(MIX)} hold for arithmetic surfaces by the works of Katz, Schaps, and
Vishne \cite{KatzSchapsVishne2007} and Selberg \cite{Selberg1965}. The first
gives the required injectivity-radius bounds, while Selberg's lower bound
$3/16$ for the spectral gap implies \textbf{(MIX)} by the work of Ratner
\cite{Ratner}.

The assumptions also hold, in a suitable quantitative form, with high
probability for random surfaces of large genus, whose spectral geometry has
seen substantial progress in recent years. We formulate this using the
Weil-Petersson random model, although analogous statements could also be
considered in other models such as random covers \cite{MageeNaudPuder}. For
background on the Weil-Petersson model and its use in spectral geometry, see,
e.g.
\cite{mirzakhaniGrowthWeilPeterssonVolumes2013,Wright2020,
AnantharamanMonk2023,gilmoreShortGeodesicLoops2021,
monkBenjaminiSchrammConvergence2022}.

Our main result can then be rephrased probabilistically as follows.

\begin{theorem}\label{Thm:propalistic_main}
    Let $\{X_g\}_{g\in\N}$ be a sequence of random compact hyperbolic surfaces, where the distribution of $X_g$ is given by the Weil-Peterson distribution on the moduli space $\mathcal{M}_g$ for all $g\in\N$, where $g$ is the genus. Suppose that $\{V_g\}_{g\in\N}$ is a sequence of potentials that satisfies the condition {\normalfont\textbf{(POT)}} with the two numbers $C_{\min}$ and $C_{\max}$. Consider for each $g\in\N$ the Schr\"odinger operator  $H_{V_g}^{(g)}= -\Delta_{X_g}+V_g$ with eigenvalues $\lambda_1^{(g)}\leq \lambda_2^{(g)}\leq \lambda_3^{(g)} \leq \lambda_4^{(g)} \cdots $ and eigenvectors $\{\psi_j^{(g)}\}_{j\in\N}$. For all $g,j\in\N$ set $\rho_{j}^{(g)} = \sqrt{\lambda_j^{(g)}-\frac{1}{4}}$. Fix $\frac{1}{4}< a' < a < b < b'$  such that $b-C_{\max}>a  -\min(0, C_{\min})$ and a sequence of spectral windows $I_g$ such that $[a,b] \subset I_g \subset [a',b']$. Then for any uniformly bounded sequence of measurable functions $\{a_g\}_{g\in\N}$ the properties (1)-(3) of Theorem~\ref{thm:main} hold with high probability as the genus goes to infinity.
\end{theorem}

A difference between the deterministic and random settings is that
\textbf{(BSC)}, \textbf{(UND)}, and \textbf{(MIX)} do not hold uniformly for
every surface. Instead, the geometric assumptions are replaced by
quantitative bounds in terms of the genus which hold with high probability,
while \textbf{(MIX)} holds with high probability for all observables in
$L^2(X_g)$. More precisely, for all sufficiently large $g$ there exists a set
$\mathcal A_g\subset\mathcal M_g$ such that
\[
    \mathcal P_g^{\mathrm{WP}}(\mathcal A_g)\longrightarrow1
    \qquad\text{as }g\to\infty,
\]
where $\mathcal P_g^{\mathrm{WP}}$ denotes Weil-Petersson probability
measure, and for surfaces sampled from $\mathcal A_g$ the quantitative
versions of the assumptions required in the proof hold. We give a short
overview of the proof of Theorem~\ref{Thm:propalistic_main} in
Section~\ref{SEC:outline_probabilistic_thm}.

\subsection{Comparison with graph methods}

Our results are closely related to quantum ergodicity and quantum mixing for Schr\"odinger operators on large graphs. In the works of Anantharaman and Sabri \cite{AnantharamanSabriAnnals,anantharamanQuantumErgodicityAnderson2017}, quantum ergodicity is established for tree-like graphs under spectral assumptions on the limiting resolvent, including the condition \textbf{(Green)} in \cite{AnantharamanSabriAnnals} and the Anderson model on trees in \cite{anantharamanQuantumErgodicityAnderson2017}. More recently, Bordenave, Letrouit, and Sabri \cite{bordenave2026quantummixinglargeschreier} proved quantum mixing for sequences of Schreier graphs using strong convergence \cite{BordenaveCollins}. See also \cite{MageeThomasZhao2023} for quantum unique ergodicity on Cayley graphs of quasirandom groups.

A major difference is that the graph arguments can exploit a detailed spectral theory of the limiting operator on the universal covering tree. For the Anderson model this goes back to Klein \cite{Klein1998}, whose proof of absolute continuity uses the recursive structure of Green's functions on trees. Such a theory can then be used to identify spectral regimes in which extended states, and consequently quantum ergodicity, should occur. No comparable theory is currently available for general Schr\"odinger operators on the hyperbolic plane. For example, if $V$ is a periodic potential obtained by placing $\varepsilon$-bumps along an orbit $\Gamma z_0$, then even basic spectral questions for $-\Delta_{\mathbb H}+V$, such as absolute continuity of the spectrum in natural regimes, appear to be largely open. Our result therefore takes a different point of view: instead of imposing spectral assumptions on a limiting Schr\"odinger operator, we give a geometric and perturbative criterion ensuring quantum ergodicity and quantum mixing directly on the finite-volume surfaces. In this sense, the condition \textbf{(POT)} provides a sufficient assumption for delocalisation which does not require any prior understanding of the spectral type of a limiting operator.

The methods show this difference. In \cite{bordenave2026quantummixinglargeschreier}, the quantum variance on a macroscopic spectral window is reduced to narrow windows and expressed through the imaginary part of the resolvent. Our argument instead uses the hyperbolic sine propagator
\[
P_{V_n}(t)=\frac{\sin\big(t\sqrt{H_{V_n}-\tfrac14}\big)}{\sqrt{H_{V_n}-\tfrac14}},
\]
whose Duhamel formula separates the free dynamics from the perturbation. The free contribution is controlled by classical mixing, while the contribution of the potential is estimated directly through its averaged $L^2$ mass. Formally, the propagator and resolvent approaches are related through the Laplace transform, with the time scale $T$ in our argument playing a role analogous to $1/\eta$ for a resolvent regularised at scale $\eta$. It would be interesting to understand whether this analogy can be made quantitative, and whether resolvent methods on the hyperbolic plane could lead to results beyond the perturbative regime considered here.

\subsection{Further questions} The comparison above also points to a natural question left open by Theorem~\ref{thm:main}: how far can the condition \textbf{(POT)} be relaxed? Our result allows potentials which remain of order one locally, but requires their $L^2$ mass to be negligible compared with the volume. At the opposite extreme, sufficiently strong or sufficiently disordered potentials are expected to produce localisation, at least in regimes where an appropriate infinite-volume operator develops pure point spectrum. 

This connects with the broader problem of Anderson localisation and delocalisation (see e.g. \cite{Klein1998,AizenmanWarzel2015,ErdosSalmhoferYau2007,Ueberschar2023,RudnickUeberschar2012,HezariRiviere2017,KurlbergUeberschar2017,KurlbergUeberschar2014} for the Euclidean setting). Between these two extremes one may expect a transition from delocalised to localised eigenstates, possibly through intermediate regimes (see e.g. see, e.g. \cite{Ueberschar2025,KeatingUeberschar2022,ueberschaer2023multifractalityintermediatequantumsystems} again in the Euclidean setting). Establishing such a transition for hyperbolic surfaces would require substantially new input, since at present even the spectral theory of natural infinite-volume hyperbolic Schr\"odinger operators is much less developed than its tree analogue. 

The geometric assumptions themselves should also play an essential role. Examples and counter-examples in the Euclidean setting include e.g. \cite{Hippi, HippiMikkelsen2026} and \cite{McKenzie2022} for graphs. It would be interesting to construct analogous examples directly on negatively curved or degenerating surfaces and to determine the sharp geometric assumptions under which large-volume quantum ergodicity and quantum mixing persist.

\subsection{Outline of the paper}

In Section~\ref{SEC:potential}, we give examples of potentials satisfying
\textbf{(POT)}, including point-cloud, induced $L^p$, and Hartree-type
potentials. In Section~\ref{SEC:pre}, we collect the required geometric and
mixing estimates. Section~\ref{SEC:Spectral_theory_schr} contains the spectral
theory and wave propagators used in the proof, while
Section~\ref{sec:Hilbert_Schmidt_norm_bounds} establishes the required
Hilbert-Schmidt bounds. The proof of Theorem~\ref{thm:main} is given in
Section~\ref{SEC:proof_main_thm}, and the probabilistic version,
Theorem~\ref{Thm:propalistic_main}, is treated in
Section~\ref{SEC:outline_probabilistic_thm}. 

\subsection*{Acknowledgements} We thank Farrell Brumley, L\'aszl\'o Erd\H{o}s, Etienne Le Masson, Cyril Letrouit, Jon Keating, Carsten Peterson, Mostafa Sabri and Joe Thomas for useful discussions and comments during the preparation of this manuscript.

\section{Examples of potentials}
\label{SEC:potential}

We now give several classes of potentials satisfying \textbf{(POT)} and hence
covered by Theorems~\ref{thm:main} and~\ref{Thm:propalistic_main}. Throughout,
let $\{X_n\}_{n\in\N}$ be a sequence of compact connected hyperbolic surfaces
satisfying \textbf{(BSC)}, \textbf{(UND)}, and \textbf{(MIX)}, and write
$X_n=\Gamma_n\setminus\mathbb H$ with fundamental domain $D_n$.

\subsection{Point-cloud potentials}

We can obtain examples by superposing local scatterers over a sparse point
set. Let $\mathcal X\subset\mathbb H$ be locally finite and assume that for
some $\delta>0$,
\[
    \dist(x_i,x_j)>2\delta
    \qquad
    \text{whenever }x_i,x_j\in\mathcal X,\quad x_i\neq x_j.
\]
The separation guarantees that bounded local scatterers give uniformly
bounded total potentials. We distinguish two regimes.

\begin{enumerate}
    \item
    Suppose
    \[
        \#(\mathcal X\cap D_n)=o(g_n),
    \]
    and let $W\in L^\infty(\mathbb H)$ satisfy
    $\supp(W)\subset B(i,\delta)$. Define
    \[
        V_n(x)
        =
        \sum_{x_j\in\mathcal X\cap D_n}
        W(T_{x_j}(x)),
    \]
    where $T_{x_j}$ maps $x_j$ to $i$. Then $(V_n)$ is uniformly bounded,
    and
    \begin{equation}
        \norm{V_n}_{L^2(X_n)}^2
        \leq
        C
        \sum_{x_j\in\mathcal X\cap D_n}
        \int_{D_n}|W(T_{x_j}(x))|^2\,dx
        \leq
        C\,\#(\mathcal X\cap D_n)\norm{W}_{L^2(\mathbb H)}^2
        =
        o(g_n).
        \label{eq:point-cloud-L2}
    \end{equation}
    Hence \textbf{(POT)} holds.

    \item
    Suppose instead that
    \[
        \#(\mathcal X\cap D_n)=\mathcal O(g_n),
    \]
    and let $W_n$ be uniformly bounded in $L^\infty(\mathbb H)$ with
    \[
        \supp(W_n)\subset B(i,\delta_n),
        \qquad
        \delta_n\to0.
    \]
    Setting
    \[
        V_n(x)
        =
        \sum_{x_j\in\mathcal X\cap D_n}
        W_n(T_{x_j}(x)),
    \]
    gives again a uniformly bounded sequence satisfying \textbf{(POT)}, by
    the same estimate as above together with the shrinking support of $W_n$.
\end{enumerate}

Thus \textbf{(POT)} includes point clouds whose number of scatterers tends to
infinity and can approach the thermodynamic scale, provided either the density
is $o(1)$ or the individual scatterers shrink appropriately.

These examples are related to the low-density, or Boltzmann-Grad, regime in
Euclidean space, where linear Boltzmann dynamics can arise as the macroscopic
transport equation; see, e.g.
\cite{EngErdoslinearboltzmann2005,
mikkelsen2023schrodingerevolutionlowdensityrandom,
BreteauxlinearBoltzmann2014,LinearBoltzmannErdosYau2000}.
In the hyperbolic setting, our result instead gives free transport in the
regime covered here, reflecting the strongly dispersive geometry of
$\mathbb H$.

\subsection{Potentials induced by a limiting potential}
\label{ex:potential}

Alternatively, we can fix a potential $V\in L^p(\mathbb H)\cap L^\infty(\mathbb H)$ for some $p>0$ in the universal cover, and define
$V_n$ by restricting $V$ to $D_n$ and extending $V\boldsymbol{1}_{D_n}$ as a
$\Gamma_n$-periodic function on $\mathbb H$. Then $\{V_n\}_{n\in\N}$ satisfies
\textbf{(POT)}.

Indeed, if $0<p\leq2$, then
\[
    \norm{V_n}_{L^2(X_n)}^2
    \leq
    \norm{V}_{L^\infty(\mathbb H)}^{2-p}
    \norm{V}_{L^p(\mathbb H)}^p,
\]
while for $p>2$, H\"older's inequality gives
\[
    \norm{V_n}_{L^2(X_n)}^2
    \leq
    \Vol(X_n)^{\frac{p-2}{p}}
    \norm{V}_{L^p(\mathbb H)}^2.
\]
Since $g_n\asymp\Vol(X_n)$, both bounds are $o(g_n)$, so \textbf{(POT)} holds.

This is a genuinely non-vanishing example: locally, the potentials $V_n$
converge to the fixed potential $V\neq0$. Conditions such as
$V\in L^p(\mathbb H)$, or stronger assumptions of Kato type, also arise
naturally in the study of Strichartz estimates and dispersive properties for
wave and Schr\"odinger equations with potential; see
\cite{Journe1991Decay,ErdoganGoldbergSchlag2008,goldberg2009strichartz,
beceanu2025dispersiveestimatesschrodingerswave,Chen2018,
AnkerPierfelice2009,AnkerPierfelice2014,AnPiVa1,AnPiVa2,
RodnianskiSchlag2004,Ionescu2000,hassani2011wave}.

We note that if the amplitude is allowed to vanish (i.e. the \textit{weak coupling limit}), the decay assumption on the underlying
potential $V$ can be relaxed. Let $V\in L^\infty(\mathbb H)$ and let
$\varepsilon_n\to0$. Define
\[
    V_n
    =
    \varepsilon_n V\boldsymbol{1}_{D_n},
\]
extended periodically to $\mathbb H$. Then $\{V_n\}_{n\in\N}$ is uniformly bounded and
\[
    \norm{V_n}_{L^2(X_n)}^2
    =
    \varepsilon_n^2
    \int_{D_n}|V(x)|^2\,dx
    \leq
    \varepsilon_n^2
    \norm{V}_{L^\infty(\mathbb H)}^2
    \Vol(X_n)
    =
    o(g_n).
\]
Hence \textbf{(POT)} holds.

This includes, for example, weakly coupled arrays of bumps along an orbit
$\Gamma z_0$ of a Fuchsian group, analogous to weak-coupling regimes studied
for the linear Boltzmann equation in Euclidean space
\cite{BreteauxlinearBoltzmann2014,LinearBoltzmannErdosYau2000}.
In this regime our theorem again yields free macroscopic transport.

\subsection{Dilute many-body Bose gases and effective Hartree potentials}

We finally describe an application arising from dilute Bose gases on large
hyperbolic surfaces. Lemm and Siebert \cite{LemmSiebertBEC2022} considered
the thermodynamic limit of the $N_n$-particle many-body Hamiltonian
\[
    H_{N_n}
    =
    \sum_{i=1}^{N_n}(-\Delta_{x_i})
    +
    \sum_{1\leq i<j\leq N_n}V(d(x_i,x_j)),
\]
acting on
$L^2_{\mathrm{sym}}(X_n^{N_n})$, where the interaction potential $V$ is
supported in $[0,R]$. Under a uniform spectral gap and a suitable diluteness
condition, they prove Bose-Einstein condensation onto the constant mode
\[
    \psi_0
    =
    \Vol(X_n)^{-1/2}\mathbf 1_{X_n}.
\]

At the Hartree level one restricts to product states
\[
    \Psi(x_1,\ldots,x_{N_n})
    =
    \phi_n(x_1)\cdots\phi_n(x_{N_n}),
    \qquad
    \norm{\phi_n}_{L^2(X_n)}=1.
\]
The corresponding Euler-Lagrange equation is
\[
    (-\Delta+U_{\phi_n})\phi_n
    =
    \lambda\phi_n,
\]
where
\[
    U_{\phi_n}(x)
    =
    (N_n-1)
    \int_{X_n}
    V(d(x,y))|\phi_n(y)|^2\,dy.
\]

Motivated by condensation onto the constant mode, we freeze the Hartree
potential at $\phi_n=\psi_0$. This gives
\[
    U_{\psi_0}(x)
    =
    \frac{N_n-1}{\Vol(X_n)}
    \int_{X_n}V(d(x,y))\,dy.
\]
If $\InjRad_{X_n}(x)>R$, then $B(x,R)$ is isometric to a hyperbolic ball, so
the integral above is independent of $x$. Hence
\[
    U_{\psi_0}
    =
    c_n+W_n,
\]
where
\[
    c_n
    :=
    \frac{N_n-1}{\Vol(X_n)}
    \int_{\mathbb H}V(d(i,z))\,dz
\]
and $W_n$ is supported on the $R$-thin part
\[
    \{x\in X_n:\InjRad_{X_n}(x)\leq R\}.
\]
By \textbf{(BSC)} and the uniform boundedness of $W_n$,
\[
    \norm{W_n}_{L^2(X_n)}^2=o(g_n).
\]
Thus $W_n$ satisfies \textbf{(POT)}.

In the thermodynamic regime $N_n\asymp\Vol(X_n)$, the constants $c_n$ remain
of order one. Since adding a constant only shifts the spectrum, the frozen
Hartree operator
\[
    H_n^{\mathrm{fr}}
    =
    -\Delta_{X_n}+U_{\psi_0}
\]
has the same eigenfunctions as
\[
    -\Delta_{X_n}+W_n,
\]
with the spectral window shifted by $c_n$. Therefore
Theorems~\ref{thm:main} and~\ref{Thm:propalistic_main} apply: in the dilute
thermodynamic regime, the excitation modes of the frozen Hartree operator are
quantum mixing in any admissible compact spectral window.

This example also indicates a natural boundary of our method. In a regime
where the condensate profile becomes genuinely non-constant, the Hartree
potential may have macroscopic $L^2$-mass,
\[
    \norm{U_{\phi_n}}_{L^2(X_n)}^2\asymp g_n,
\]
and \textbf{(POT)} no longer applies. Extending quantum ergodicity to such a
regime would require information beyond $L^2$-smallness, possibly through the
spectral theory of limiting Schr\"odinger operators on $\mathbb H$.

\section{Geometric preliminaries}\label{SEC:pre}

\subsection{Hyperbolic surfaces}
We will by $\mathbb{H}$ denote the hyperbolic plane and we will identify $\bH$ with the upper half plane, that is
\begin{equation*}
    \bH = \big\{ z=x+iy \in \C \,| \,y>0 \big\}
\end{equation*}
equipped with the Riemannian metric
\begin{equation*}
    ds^2 = \frac{dx^2+dy^2}{y^2}.
\end{equation*}
For a point $z \in \bH$ and a tangent vector $v \in T_z\bH$, we will also write $\norm{v}_z$ for its norm with this metric.
We will write $d(z,z')$ for the distance between two points $z,z'\in \bH$. Moreover, the hyperbolic volume is given by
\begin{equation*}
    dz = \frac{1}{y^2} d x d y.
\end{equation*}
The orientation-preserving isometry group of $\bH$ can be identified with the group $\PSL(2,\R)$, consisting of real $2\times2$ matrices of determinant 1 modulo $\pm \id$. The group $\PSL(2,\R)$ acts on $\bH$ via M\"obius transformations. That is, for $\gamma\in \PSL(2,\R)$ with
\begin{equation}\label{EQ:Pre_gamma_def}
    \gamma=\begin{pmatrix}
        a &b \\ c &d
    \end{pmatrix}
\end{equation}
the element $\gamma z$ for any $z\in \bH$ is given by
\begin{equation*}
    \gamma z = \frac{az +b}{cz +d}.
\end{equation*}
The unit tangent bundle of $\bH$ is 
\begin{equation*}
    T^1\bH = \big\{ (z,v) \in \bH \times T_z\bH \,|\, \norm{v}_z=1 \big\}.
\end{equation*}
The group $\PSL(2,\R)$ also acts on $T^1\bH$ through the action
\begin{equation*}
    \gamma (z,v) = \Big( \tfrac{az+b}{cz+d},\tfrac{v}{(cz+d)^2} \Big),
\end{equation*}
for any $(z,v)\in T^1\bH$ and $\gamma \in \PSL(2,\R)$, where $\gamma$ is again given by \eqref{EQ:Pre_gamma_def}. It can be verified that the two spaces $\PSL(2,\R)$ and $T^{1}\bH$ can be identified. This can be done, for example, by evaluating the action at the point $(i,i)$ to show that $\PSL(2,\R)$ is homeomorphic to $T^{1}\bH$. This leads to the identification that a point $(z,v)\in T^1\bH$ is identified with $\gamma\in \PSL(2,\R)$ such that $\gamma (i,i) = (z,v)$.

A compact hyperbolic hyperbolic surface can be seen as the quotient $ X=\Gamma\setminus \bH$ of $\bH$, where $\Gamma\in\PSL(2,\R)$ is a discrete subgroup consisting solely of hyperbolic elements and the identity. That is, for all $\gamma\in\Gamma\setminus\{\id\}$ we have $|\Tr[\gamma]|>2$. These groups are called Fuchsian groups. Given a point $z$ in $\bH$, we can associate the Dirichlet fundamental domain $D$ to $X$ defined by
\begin{equation*}
    D = D_z := \big\{ z'\in \bH : d(z,z') < d(z,\gamma z'), \, \forall \, \gamma \in \Gamma\setminus\{\id\}\big\}.
\end{equation*}
We will, if nothing else is assumed, suppose that the fundamental domain is generated around the point $i$ in $\bH$.

The injectivity radius of $X = \Gamma \setminus \bH$ at a point $z\in X $ is given by
\begin{equation*}
    \InjRad_X(z) = \frac{1}{2} \min_{\gamma\in\Gamma\setminus\{\id\}} d(z,\gamma z),
\end{equation*}
and the global injectivity radius of the surface $X$ is
\begin{equation*}
    \InjRad_X = \inf_{z\in X} \InjRad_X(z).
\end{equation*}
In our analysis, we will need to split the surface into two parts, the one where the injectivity radius is large and the other where it is small: for $T>0$ we will use the notation
\begin{equation*}
    X(> 2T) \coloneqq \{ x\in X \,|\, \InjRad_X(x) \geq 2T \} \qquad\text{and}\qquad X(\leq 2T) = X\setminus X(> 2T).
\end{equation*}
We will use similar notation for the corresponding parts of the surface in the fundamental domain. 

\subsection{Lattice counting}

In the proof, we will need some a priori lattice counting estimates. This is contained in the following lemma, which is a slightly modified version of \cite[Lemma 5]{Marklof_2011}, and the proof is based on the same ideas. The difference is that we are giving explicit constants. 
\begin{lemma}\label{LE:counting_loop_lemma}
Let $X=\Gamma\setminus \bH$, where $\Gamma$ is a strictly hyperbolic Fuchsian group. Then for any
$
0<r\leq \min\!\left(1,\frac{\InjRad_X}{2}\right)$ and $t\geq 0$ one has
 \begin{equation*}
 \sup_{(x,y)\in \bH\times \bH}
 \sharp\{\gamma\in\Gamma\mid d(x,\gamma y)\leq t\}
 \leq
 \frac{e^{t+1}}{r^2}.
 \end{equation*}
\end{lemma}
\begin{proof}
Let $x,y\in \bH$ be two points in the hyperbolic plane. 
Due to our assumptions on $r$, it follows that $B(\gamma y,r)\cap B(\gamma' y, r) =\emptyset$ for all $\gamma,\gamma'\in\Gamma$ and 
\begin{equation}\label{EQ:counting_loop_lemma_1}
    \bigcup_{\substack{\gamma \in \Gamma \\ d(x,\gamma y)\leq t}} B(\gamma y,r) \subset B(x,t+r).
\end{equation}
Now we set
\begin{equation*}
   M_{xy} =  \sharp \big\{ \gamma \in \Gamma \, \big| \, d(x,\gamma y) \leq t \big\}.
\end{equation*}
Using \eqref{EQ:counting_loop_lemma_1} and that $\Vol(B(\gamma y,r)) = 4\pi \sinh^2(\frac{r}{2})$ for all $\gamma \in \Gamma$ we have
\begin{equation*}
   4\pi e^{t+r} \geq 4\Vol(B(x,t+r)) \geq M_{xy} 16\pi \sinh^2(\frac{r}{2}) \geq M_{xy} 4\pi e^{r-1} r^2.
\end{equation*}
This implies the bound
\begin{equation*}
    M_{xy} \leq \frac{e^{t+1}}{r^2},
\end{equation*}
where we have used our assumptions on $r$. Since the upper bound is independent of $x$ and $y$, this implies the stated estimate and concludes the proof.
\end{proof}

\subsection{Mixing of the geodesic flow} \label{sec:mixinggeo}

The geodesic flow $\varphi_t$ on $\mathbb{H}$ is a one-parameter family of maps
\begin{equation*}
    \varphi_t : T^{1}\bH \mapsto T^{1}\bH, \qquad t\in\R
\end{equation*}
such that for any point $(z,v)\in T^1\bH$ the curve $\varphi_t(z,v) = (\eta(t),\eta'(t))$ is the unique geodesic parameterised by the arc length that satisfies $\eta(0)=z$ and $\eta'(0) =v$. Under the identification of $T^{1}\bH$ with $\PSL(2,\R)$ we can parameterise the geodesic flow as a right multiplication in $\PSL(2,\R)$  by the one parameter subgroup 
\begin{equation*}
    \Bigg\{ \begin{pmatrix}
        e^{t/2} &0 \\ 0 &e^{-t/2}
    \end{pmatrix} \,\Bigg| \, t\in\R \Bigg\}.    
\end{equation*}

Given a point $z$ in $\bH$, we can parameterise any other point $z'$ in $\bH$ using the geodesic flow. We do this in the usual way by fixing some direction so that it has angle 0 and then any point $z'$ in $\bH$ can be written as
\begin{equation*}
    z' = \pi_1(\varphi_r(z,\theta)),
\end{equation*}
where $r = d(z,z')$, $\pi_1$ is the projection onto the first coordinate, and $\theta$ is the angle depending on the fixed direction and $z'$. Using this, we can define the polar coordinates in $\bH$ around $z$ and the change of variables $z'\mapsto (r,\theta)$ transforms the metric and the volume element as follows
\begin{equation*}
    ds^2 = dr^2 +\sinh^2(r) d \theta^2 \qquad\text{and}\qquad dz = \sinh(r) dr d\theta.
\end{equation*}

Let us now fix a compact hyperbolic surface $X = \mathrm{SO}(2,\R) \setminus \mathrm{SL}(2,\R)/\Gamma $, where $\Gamma$ is a lattice in $\mathrm{SL}(2,\R)$. As a Riemannian surface $X$, we have the canonical unit-cotangent bundle $T^1 X$ consisting formally of pairs $x \in X$ and directions $\theta \in S_x^* X$, which we can identify with the unit circle $\S^1$. Then the unit co-tangent bundle of $X$ can be identified with $T^1 X = \mathrm{SL}(2,\R)/\Gamma,$ where also geodesic flow $\phi_t : T^1 X \to T^1 X$ is defined, and for which the unique absolutely continuous invariant measure \textit{Liouville measure} $dx \, d\xi$ projects onto the hyperbolic volume measure $dx$ on $\mathbb{H}$.

A key tool in the proof of Theorem \ref{thm:main} is the use of the property \textbf{(MIX)} for the geodesic flow. On compact hyperbolic surfaces $X$ we can always bound the rate of mixing using the spectral gap $\lambda_1(X)$. These were first proved by Ratner \cite{Ratner} who connected the spectral gap to the rate of exponential mixing of the geodesic flow via representation theoretic approach, but the quantitative bound we use here comes from the work of Matheus \cite{Matheus}, in particular that the multiplicative constant is independent of the surface:

\begin{theorem}
\label{Thm: Exponential mixing}
Let $X$ be a compact connected hyperbolic surface with geodesic flow $ \varphi_{t} $ and with $\lambda_1( X )$ being the smallest eigenvalue of the Laplace-Beltrami operator $ -\Delta_X $. Let $f,g \in L^2(X)$. Then for all $t \ge 0$
\[
\bigg| \big\langle f \circ \varphi_t, g \big\rangle_{L^2(T^1X)} 
- \int_{T^1X} f \int_{T^1X} g \bigg|
\leq 11 \, e^{ \beta( \lambda_{1}( X ) ) } \, ( 1 + t ) \, e^{- \beta(\lambda_1( X ) ) t} \, 
\|f\|_2 \, \|g\|_2,
\]
where
\begin{align*}
	&
	\beta( \lambda ) = 
	\begin{cases}	
		1 - \sqrt{1 - 4\lambda}, \qquad & \text{ if } \lambda \leq \frac{1}{4},
		\\
		1, \qquad & \text{ if } \lambda > \frac{1}{4},
	\end{cases}
\end{align*}
and where $f(x) = f(x,\theta)$ and $g(x) = g(x,\theta)$ are viewed as functions on $T^1X$, the unit tangent bundle.
\end{theorem}

Thus \textbf{(MIX)} holds for any $\gamma > 0$ due to the exponential rate. However, as Theorem \ref{thm:main} only needs polynomial rate of mixing, this allows us to also consider sequences of surfaces where the spectral gap $\lambda_1(X_n) \to 0$ as $n \to \infty$.

Furthermore, note that in this statement, the observables $f, g$ are in $L^2(X)$ and then lifted to $L^2(T^1X)$. This is why the error term only has $L^2$ norms. There is a more general (and usual) form in which the observables are defined on phase space $T^1 X$, and then the error term can be given in terms of Sobolev norms of the observables. This can probably be used to give a phase space version of our result along the lines of Abért-Bergeron-Le Masson \cite{abertEigenfunctionsRandomWaves2018}, see also the recent work of Anantharaman and Saha \cite{AnantharamanSaha} but we do not pursue this in this work.

\section{Spectral theory of Schr\"odinger operators}\label{SEC:Spectral_theory_schr}

\subsection{Definition of operators, self-adjointness and large scale spectral bounds}\label{SEC:Kernels_def_selfadj}

Given a connected hyperbolic surface $X$ we denote by $-\Delta_X$ the standard Laplace-Beltrami operator acting in $L^2(X)$, defined in the usual way. This is a positive self-adjoint operator with domain $\mathcal{D}(-\Delta_X) = H^2(X)$, the second order Sobolev space. By the Kato-Rellich theorem it follows that for all real valued $V\in L^\infty(X)$ the Schr\"odinger operator 
\begin{equation*}
    H_V = -\Delta_X +V,
\end{equation*}
 is well defined, self-adjoint, and has domain $\mathcal{D}(H_V) = H^2(X)$. For details, see, e.g. \cite{ReedSimonII}. Moreover, we have that if the hyperbolic surface $X$ is compact, then $-\Delta_X$ and $H_V$ for any real valued $V\in L^\infty(X)$ have a purely discrete spectrum. For the case where $X=\bH$ we have that $-\Delta_\bH$ have purely absolute continuous spectrum in $[\frac{1}{4},\infty)$. Moreover, according to the Kato–Rellich theorem, for any real valued $V \in L^\infty(\bH)$, the essential spectrum of $H_V$ is $[\frac{1}{4},\infty)$. In particular, note that we do not a priori know that $[\frac{1}{4},\infty)$ is an absolute continuous spectrum. Self-adjointness of the operators gives us access to the functional calculus for the operator, and hence for any measurable function $f$ the (possible unbounded) operator $f(H_V)$ is well defined; see, e.g., \cite{ReedSimonI}. We will here be interested in three different functions of our operators. Two of them are defined in the following subsection. The third function corresponds to the spectral projections. For a given interval $I$ we will use the notation $\Pi_I(H_V)$ to denote the spectral projection of $H_V$ onto $I$. That is,
\begin{equation*}
    \Pi_I(H_V) = \boldsymbol{1}_I(H_V).
\end{equation*}

In our setting, we are considering a sequence of surfaces that Benjamini-Schramm converges to the hyperbolic plane. For our approach here, it is important to know how the number of eigenvalues in a fixed window $I = [a,b]\subset(\frac{1}{4},\infty)$ behaves under this limit. For the case of the Laplace-Beltrami operator on a sequence of compact connected hyperbolic surfaces that Benjamini-Schramm converges to the hyperbolic plane we have from \cite[Lemma 9.1]{lemassonQuantumErgodicityBenjaminiSchramm2017} the following large scale Weyl law result.     
    
\begin{lemma}\label{LE:Weyl_laplace}
Let $\{X_n\}_{n\in\N}$ be a sequence of compact connected hyperbolic surfaces that satisfy the property {\normalfont\textbf{(BSC)}}. For any compact interval $I\subset (\frac{1}{4},\infty)$, we have 
\begin{equation}
    \lim_{n\rightarrow\infty} \frac{\mathcal{N}(-\Delta_{X_n},I)}{\Vol(X_n)} = \frac{1}{4\pi} \int_\R \boldsymbol{1}_I (\tfrac{1}{4} +r^2) \tanh(\pi r) r \, d r,
\end{equation}
where $\mathcal{N}(-\Delta_{X_n},I)$ denotes the number of eigenvalues of $-\Delta_{X_n}$ in the interval $I$. 
\end{lemma}

Using this result, we have the following asymptotic lower bound for the number of eigenvalues for a sequence of Schr\"odinger operators in a sequence of compact intervals $I_n$ containing a fixed spectral window $I$.

\begin{lemma}[Lower bound for counting function]\label{LE:Weyl_schrodinger}
Let $\{X_n\}_{n\in\N}$ be a sequence of compact connected hyperbolic surfaces that satisfy the property {\normalfont\textbf{(BSC)}} and let $\{V_n\}_{n\in\N}$ be a sequence of potentials satisfying {\normalfont\textbf{(POT)}}. For any compact interval $I_n = [a_n,b_n]\subset (\frac{1}{4},\infty)$ containing a fixed $I = [a,b] \subset I_n$ for all $n \in \N$, such that $b-C_{\max} >a  -\min(0, C_{\min})$, there exist a positive constant $c$ and a natural number $N\in\N$ such that 
\begin{equation}
    \mathcal{N}(H_{V_n},I_n) \geq c\Vol(X_n) \qquad\text{for all $n\geq N$.} 
\end{equation}
The constant $c$ and $N$ depends on the numbers $a$, $b$, $C_{\max}$, and $C_{\min}$.
\end{lemma}

\begin{proof}
In this proof, we will let $\widetilde{C}_{\min}= \min(0,C_{\min})$. Due to our assumptions on the sequence of potentials, we have for all $n\in\N$
\begin{equation*}
    -\Delta_{X_n}+\widetilde{C}_{\min} \leq H_{V_n} \leq  -\Delta_{X_n}+ C_{\max}
\end{equation*}
in the sense of quadratic forms. By the min-max Theorem (see, e.g. \cite[Chapter XIII]{ReedSimonIV}), this implies for all $\alpha\in\R$ that
\begin{equation*}
    \Tr\big[ \boldsymbol{1}_{(-\infty, \alpha)}(-\Delta_{X_n}+ C_{\max}) \big] \leq \Tr\big[ \boldsymbol{1}_{(-\infty, \alpha)}(H_{V_n}) \big] \leq \Tr\big[ \boldsymbol{1}_{(-\infty, \alpha)}(-\Delta_{X_n}+ \widetilde{C}_{\min}) \big].
\end{equation*}
The same inequalities are true if $\alpha$ is included in the interval. Using these, we have that
\begin{equation*}
    \begin{aligned}
    \Tr\big[ \boldsymbol{1}_{[a,b]}(H_{V_n}) \big] 
    ={}& \Tr\big[ \boldsymbol{1}_{(-\infty,b]}(H_{V_n}) \big] - \Tr\big[ \boldsymbol{1}_{(-\infty,a)}(H_{V_n}) \big] 
    \\
    \geq{}& \Tr\big[ \boldsymbol{1}_{(-\infty,b]}(-\Delta_{X_n}+ C_{\max}) \big] - \Tr\big[ \boldsymbol{1}_{(-\infty,a)}(-\Delta_{X_n}+ \widetilde{C}_{\min}) \big]
    \\
    ={}& \Tr\big[ \boldsymbol{1}_{(-\infty,b- C_{\max}]}(-\Delta_{X_n}) \big] - \Tr\big[ \boldsymbol{1}_{(-\infty,a-\widetilde{C}_{\min})}(-\Delta_{X_n}) \big]
    \\
    ={}& \Tr\big[ \boldsymbol{1}_{[a-\widetilde{C}_{\min},b- C_{\max}]}(-\Delta_{X_n}) \big].
    \end{aligned}
\end{equation*}
We set $\tilde{I} = [a-\widetilde{C}_{\min},b- C_{\max}]$. Due to our assumptions, we have $\tilde{I} \subset (\frac{1}{4},\infty)$.  Since $\Tr\big[ \boldsymbol{1}_{\tilde{I}}(-\Delta_{X_n}) \big] = \mathcal{N}(-\Delta_{X_n},\tilde{I})$ it follows from Lemma~\ref{LE:Weyl_laplace} that there exists a positive constant $c$ and $N\in \N$ such that 
\begin{equation*}
    \Tr\big[ \boldsymbol{1}_{[a-\widetilde{C}_{\min},b- C_{\max}]}(-\Delta_{X_n}) \big] \geq c \Vol(X_n) \qquad\text{for all $n\geq N$,}
\end{equation*}
where the constant $c$ and $N$ depends on the numbers $a$, $b$, $C_{\max}$ and $\widetilde{C}_{\min}$. Combining this with the estimate obtained above, we have the following
\begin{equation*}
    \Tr\big[ \boldsymbol{1}_{[a,b]}(H_{V_n}) \big] \geq c \Vol(X_n) \qquad\text{for all $n\geq N$.}
\end{equation*}
This concludes the proof.
\end{proof}

\subsection{Propagators}\label{SEC:propagators}
As usual in these kinds of results we will introduce a propagator to access dynamical information. Let $X$ be a connected hyperbolic surface, and let $V\in L^\infty(X)$ be real valued and consider the Schr\"odinger operator $H_V = -\Delta +V$.  Then the propagator we will use here is the following wave propagator
\begin{equation*}
    P_V(t) = \big(H_V-\tfrac{1}{4}\big)^{-\frac12}\sin\big(t\sqrt{H_V-\tfrac14}\big),
\end{equation*}
where the subtraction of $\frac14$ is a spectral shift, which is introduced since the spectrum of $-\Delta_\bH$ is $[\frac14,\infty)$. We can define these propagators using the functional calculus. One thing to note here is that $P_V(t)$ is self-adjoint and bounded for all $t\geq 0$. To see this, let $h: \R_+\times \R \mapsto \R$ be given by
\begin{equation}\label{EQ:def_h}
    h(t,\lambda) = \big(\lambda-\tfrac{1}{4}\big)^{-\frac12}\sin\big(t\sqrt{\lambda-\tfrac14}\big).
\end{equation}
That $h$ is indeed real valued can be seen by taking $\lambda_0<\frac14$ and computing that
\begin{equation*}
     h(t,\lambda_0) = \big(\tfrac{1}{4} - \lambda_0 \big)^{-\frac12}\sinh\big(t\sqrt{\tfrac14 - \lambda_0}\big).
\end{equation*}
Moreover, it is also easily seen that when restricted to $\R_+\times [\alpha,\infty)$, the function $h$ is bounded, for any $a \in \R$. Since our operator $P_V(t) = h(t,H_V)$ and $H_V$ is self-adjoint and semi-bounded from below, we see that $P_V(t)$ is a self-adjoint bounded operator.

The intuition behind calling this a propagator is that for a suitable function $\varphi$, then the function $u(t,x)=P_V(t) \varphi$ is a solution to the wave equation $\partial_t^2u = (H_V -\frac{1}{4}) u$ with the initial conditions $u(0,x) = 0$ and $\partial_t u(t,x)|_{t=0} = \varphi(x)$.

It would be desirable to have good expressions for the kernel of the operator $ P_V(t)$. However, this is not feasible when the potential $V$ is just a general bounded function with real values. What we have instead is that the propagator $P_V(t)$ satisfies a Duhamel formula. For the sake of completeness, we will include the formula as a lemma and also prove it.  
\begin{lemma}
Let $X$ be a connected hyperbolic surface and $H_V = -\Delta_X + V$ be a Schr\"odinger operator with $V\in L^\infty(X)$ with the convention that $H_0=-\Delta_X$ is the Laplace-Beltrami operator. Consider the wave propagators 
\begin{equation*}
    P_V(t) = h(t,H_V) \qquad\text{and}\qquad  P_0(t) = h(t,H_0),
\end{equation*}
where $h$ is given by \eqref{EQ:def_h}. Then for any $t>0$ we have 
\begin{equation}\label{Duhamel_formula_main}
    P_V(t) = P_0(t) - Q_V(t),
\end{equation}
where 
\begin{equation}\label{Duhamel_formula_main_Q}
    Q_V(t) =  \int_{0}^{t} P_V(t_1) V P_0(t-t_1) \, dt_1.
\end{equation}
\end{lemma}

\begin{proof}
  To establish this formula, we will need the following axillary operator $S_{i}(t)$ for $i\in\{0,V\}$ given by
\begin{equation*}
    S_i(t) =\cos \big(t\sqrt{H_i-\tfrac14}\big).
\end{equation*}
By functional calculus, this is again a well-defined operator. For these operators, we have the following relations
$$ 
\partial _t P_i(t) = S_i(t) \qquad\text{and}\qquad \partial_t S_i(t) = - (H_i-\tfrac14) P_i(t), 
$$
where $i\in\{0,V\}$. Using these relations, we obtain from the fundamental theorem of calculus that 
\begin{align*}
    P_i(t) &= \int_{0}^t S_i (t_1) \, dt_1, \qquad\text{for $i\in\{0,V\}$}
    \\
    \shortintertext{and}
     S_V(t) &= S_0(t)+  \int_{0}^t S_V(t_1) (H_0-\tfrac14) P_0(t-t_1) -P_V(t_1) (H_V-\tfrac14) S_0(t-t_1) \, dt_1.
\end{align*}
where we have used $P_i(0) = 0$ and $S_i(0)=1$ for $i\in\{0,V\}$. Combining these expressions yields the following identity
\begin{equation}\label{Duhamel_step_1}
    \begin{aligned}
       P_V(t)
         ={}&\int_{0}^{t} S_0(t_2) + \int_{0}^{t_2} S_V(t_1) (H_0-\tfrac14) P_0(t_2-t_1) -P_V(t_1) (H_V-\tfrac14) S_0(t_2-t_1) \, dt_1 dt_2
        \\
        ={}& P_0(t)+ \int_{0}^{t} \int_{0}^{t_2} S_V(t_1) (H_0-\tfrac14) P_0(t_2-t_1) - P_V(t_1)( H_V-\tfrac14) S_0(t_2-t_1) \,  dt_1 dt_2.
    \end{aligned}
\end{equation}
What remains to establish \eqref{Duhamel_formula_main} is to prove that the double integral equals $-Q_V(t)$. To establish this, we first notice that by the fundamental theorem of calculus we also have the following 
\begin{equation}\label{Duhamel_step_2}
     \begin{aligned}
       P_V(t)
         ={}&\int_{0}^{t}\partial_{t_1} \big[  P_V(t_1) S_0(t-t_1)\big]  \, dt_1
        \\
        ={}& \int_{0}^{t}  S_V(t_1) S_0(t-t_1)  \, dt_1 + \int_{0}^{t} P_V(t_1) (H_0-\tfrac14) P_0(t-t_1)  \, dt_1.
    \end{aligned}
\end{equation}
Returning to the double integral in \eqref{Duhamel_step_1} we notice that only one of the operators in each term depends on the variable $t_2$. Hence, by integrating we get
\begin{equation*}
    \begin{aligned}
       \MoveEqLeft \int_{0}^{t} \int_{0}^{t_2} S_V(t_1) (H_0-\tfrac14) P_0(t_2-t_1) - P_V(t_1) (H_V-\tfrac14) S_0(t_2-t_1) \,  dt_1 dt_2
       \\
       ={}& \int_{0}^{t} S_V(t_1) \,dt_1 - \int_0^t S_V(t_1) S_0(t-t_1) \, dt_1  - \int_{0}^{t} P_V(t_1) H_V P_0(t-t_1) \, dt_1 
       \\
       ={}&P_V(t) - \int_0^t S_V(t_1) S_0(t-t_1) \, dt_1  - \int_{0}^{t} P_V(t_1) (H_V-\tfrac14) P_0(t-t_1) \, dt_1 
       \\
       ={}& \int_{0}^{t} P_V(t_1) (H_0-\tfrac14) P_0(t-t_1)  \, dt_1 - \int_{0}^{t} P_V(t_1) (H_V-\tfrac14) P_0(t-t_1) \, dt_1 
       \\
       ={}& -\int_{0}^{t} P_V(t_1) V P_0(t-t_1)  \, dt_1,
    \end{aligned}
\end{equation*}
where we have used \eqref{Duhamel_step_2}. This establishes that the double integral is indeed equal to $-Q_V(t)$, and therefore we have established the Duhamel formula \eqref{Duhamel_formula_main}.   
\end{proof}

With this formula, it will suffice for later estimates to have a good expression for the kernel of $P_0(t)$. This is the content of the following Lemma which is taken from \cite{Hippi}, where it is Proposition~{6.1}. 
\begin{lemma}\label{LE:kernel_of_unper_pro}
    Let $X=\Gamma\setminus \bH$, where $\Gamma$ be a strictly hyperbolic Fuchsian group, and let $-\Delta_X$ be the Laplace-Beltrami operator. Set
    \begin{equation*}
        P_0(t) = h(t,-\Delta_X).
    \end{equation*}
    Then the integral kernel of $P_0(t)$ is given by
    \begin{equation*}
        K^\Gamma_{t}(x,y)= \sum_{\gamma \in \Gamma} A(t,d(x,\gamma y)) =\frac{1}{2\sqrt{2} \pi} \sum_{\gamma \in \Gamma} \frac{\boldsymbol{1}(t> d(x,\gamma y))}{\sqrt{\cosh(t) - \cosh(d(x,\gamma y))}}.
    \end{equation*}
\end{lemma}
The function $A(t,r)$ from the above lemma is the Abel kernel and we will refer to it by that name in what follows. For later applications, the following small lemma on integrals in time for the function $h(t,\lambda)$ will be useful.

\begin{lemma}\label{LE:lower_bound_h(t,lambda)}
    Let $h:\R_+\times \R \mapsto \R$ be given by \eqref{EQ:def_h} and let $I=[a,b]\subset (\frac14,\infty)$ be a compact interval. Then there exists $T_0>0$ depending on $a$ and $b$ such that
    \begin{equation*}
        \inf_{\lambda\in I}\frac{1}{T} \int_0^T h^2(t,\lambda_0) \, d t \geq \frac{1}{3(b -\frac{1}{4})} \qquad\text{for all $T\geq T_0$}.
    \end{equation*}
\end{lemma}

\begin{proof}
The proof is a simple calculation, where we will use the identity $\sin^2(t\theta) = \frac{1-\cos(2t\theta)}{2} $. For $\lambda>\frac{1}{4}$ we have
    \begin{equation*}
    \begin{aligned}
    \frac{1}{T} \int_0^T h^2(t,\lambda) \, d t  
    &=\frac{1}{T} \int_0^T \frac{\sin^2(t \sqrt{\lambda - \frac{1}{4}})}{(\lambda -\frac{1}{4})} \, d t 
    =  \frac{1}{T} \int_0^T \frac{1- \cos\big(2t\sqrt{\lambda - \frac{1}{4}})}{2(\lambda -\frac{1}{4})}  \, d t
    \\
    &=  \frac{1}{2(\lambda -\frac{1}{4})} -   \frac{\sin\big(4T\sqrt{\lambda - \frac{1}{4}})}{2T(\lambda - \frac{1}{4})^{\frac32}}  \geq \frac{1}{2(\lambda -\frac{1}{4})} -   \frac{1}{4T(\lambda - \frac{1}{4})^{\frac32}}. 
    \end{aligned}
\end{equation*}
This inequity gives us the following lower bound 
\begin{equation*}
        \inf_{\lambda\in I}\frac{1}{T} \int_0^T h^2(t,\lambda_0) \, d t \geq \frac{1}{2(b -\frac{1}{4})} -   \frac{1}{4T(a - \frac{1}{4})^{\frac32}} \geq \frac{1}{3(b -\frac{1}{4})} \qquad\text{for all $T\geq T_0$,}
\end{equation*}
where $T_0$ depends on the values of $a$ and $b$. This completes the proof. 
\end{proof}

The propagator discussed so far is the one we will use to prove point (1) Theorem~\ref{thm:main}. To prove points (2) and (3) in Theorem~\ref{thm:main} we will need to slightly modify one of the propagators used. This is the same modified propagator used in \cite{Hippi}. For a given $\tau \in \R$ we define the function $\tilde{h}_\tau \colon\R_+\times \R \to \R $ by 
\begin{equation}\label{EQ:def_tilde_h}
    \tilde{h}_{\tau}(t,\lambda) = \cos(\tau t) h(t,\lambda) = \cos(\tau t)  \big(\lambda-\tfrac{1}{4}\big)^{-\frac12}\sin\big(t\sqrt{\lambda-\tfrac14}\big).
\end{equation}
As above, we then define the family of operators $\widetilde{P}_V(t,\tau)$ by 
\begin{equation*}
    \widetilde{P}_V(t,\tau) = \tilde{h}_{\tau}(t,H_V) = \cos(\tau t) P_V(t).
\end{equation*}
From Lemma~\ref{LE:kernel_of_unper_pro} we also obtain a kernel expression for $\widetilde{P}_0(t,\tau)$ by simply multiplying the kernel of $P_0(t)$ by $\cos(\tau t)$. Moreover, since $\widetilde{P}_V(t,\tau) = \cos(\tau t) P_V(t)$ the operator $\widetilde{P}_V(t,\tau)$ will also satisfy a Duhamel formula
\begin{equation}\label{Duhamel_formula_main_tilde}
    \widetilde{P}_V(t,\tau) = \widetilde{P}_0(t,\tau) - \widetilde{Q}_V(t,\tau),
\end{equation}
where $\widetilde{P}_0(t,\tau) = \tilde{h}_\tau(t,-\Delta_X)$ and
\begin{equation}
    \widetilde{Q}_V(t,\tau) =   \cos(\tau t)\int_{0}^{t}  P_V(t_1) V P_0(t-t_1) \, dt_1.
\end{equation}
We conclude this subsection with the following small lemma on integrals in time for products of functions $h(t,\lambda)$ and $\tilde{h}_\tau(t,\lambda)$.    

\begin{lemma}\label{LE:lower_bound_tilde_h_and_h(t,lambda)}
    Let $\tau\in \R$ and let the functions $h,\tilde{h}_\tau:\R_+\times \R \mapsto \R$ be given by \eqref{EQ:def_h} and \eqref{EQ:def_tilde_h} respectively. Assume that $\frac{1}{4}<m<a,b < \infty$ and that 
    \begin{equation*}
        \sqrt{a-\tfrac14}-\sqrt{b-\tfrac{1}{4}}-\tau \in(-\delta,\delta),
    \end{equation*}
    where $\delta \in (0,\frac{2}{9}\sqrt{m-\frac14})$. Then for $T = \frac{\pi}{2\delta}$ it holds
    \begin{equation*}
        \frac{1}{T}\Bigg| \int_0^T \tilde{h}_\tau(t,a) h(t,b) \, d t \Bigg|^{-1} < 8\pi  \sqrt{a-\tfrac14}\sqrt{b-\tfrac{1}{4}}.
    \end{equation*}
\end{lemma}
A proof of this lemma can be found in \cite{Hippi}, where it is Proposition~{4.1}.

\section{Quantitative bounds on Hilbert-Schmidt norms}\label{sec:Hilbert_Schmidt_norm_bounds}
In this section, we establish a number of estimates on Hilbert-Schmidt norms needed in our proof of the two main results. However, before considering the Hilbert-Schmidt norms, we will need some a priori integral estimates. Then for the estimates on the Hilbert-Schmidt norms we will firstly prove estimates that enable localisation to the ``thick'' part of the surface. Then we will establish estimates for the terms in which the potential appears. Next, we will use exponential mixing as in \cite{Hippi} to estimate the main free term. Finally, we will combine all three estimates into a single quantitative estimate needed for our analysis.  

\subsection{A priori integral bounds}
The first integral bound is a slight modification of \cite[Proposition 7.2]{Hippi}.

\begin{lemma}\label{LE:Two_abel_kernel_int_est}
Let $z,z'\in\bH$ and $t,t'\in \R_{+}$. Then we have the following
\begin{equation*}
    \int_{\bH} A(t,d(z,x))A(t',d(z',x)) \, d x \leq \frac{ \boldsymbol{1}_{[0,t+t']}(d(z,z'))}{\sqrt{\sinh (\max(|t-t'|,d(z,z')))}},
\end{equation*}
where $A(\tau,r)$ is the Abel kernel introduced in Lemma~\ref{LE:kernel_of_unper_pro}.
\end{lemma}
The difference between this lemma and \cite[Proposition 7.2]{Hippi} is only in the constant. This is because here we have the factor $(2\sqrt2 \pi)^{-2}$ on the left-hand side. The proof is done by carefully using the hyperbolic geometric and polar coordinates.    

\begin{lemma}\label{LE:Two_sqr_sinh_singu_1}
Let $z,z'\in\bH$ and $t,t'\in \R_{+}$. Then we have the following
\begin{equation*}
    \int_{\bH} \frac{\boldsymbol{1}_{[0,t]}(d(x,z))\boldsymbol{1}_{[0,t']}(d(x,z'))}{\sqrt{\sinh(d(x,z)) }\sqrt{\sinh(d(x,z')) }} \, d x \leq  \boldsymbol{1}_{[0,t+t']}(d(z,z')) 4\pi \min(t,t').
\end{equation*}
\end{lemma}

\begin{proof}
 By symmetry, we may without loss of generality assume that $t'\leq t$. Moreover, due to the presence of the two indicator functions, we see that the integral is zero if $d(z,z')>t+t'$.

Suppose that $z=z'$. Then by switching to polar coordinates centred at z we obtain the following 
\begin{equation}\label{EQ:Two_sqr_sinh_singu_1}
    \begin{aligned}
        \int_{\bH} \frac{\boldsymbol{1}(d(x,z)<t')}{\sinh(d(x,z)) } \, d x = 2\pi   \int_0^{t'} 1 \, d r = 2\pi t'. 
    \end{aligned}
\end{equation}

Suppose that $0<d(z,z') \leq 2t'$. We first extend the area of integration to the entire ball $B(z',t')$ and split this integral into two parts depending on whether $x$ is closer to $z$ or $z'$.
\begin{equation}\label{EQ:Two_sqr_sinh_singu_2}
    \begin{aligned}
        \int_{\bH} \frac{\boldsymbol{1}(d(x,z)<t)\boldsymbol{1}(d(x,z')<t')}{\sqrt{\sinh(d(x,z)) }\sqrt{\sinh(d(x,z')) }} \, d x  
        \leq{}& \int_{B(z',t')} \frac{1}{\sqrt{\sinh(d(x,z)) }\sqrt{\sinh(d(x,z')) }} \, d x  
        \\
        \leq{}&\int_{D_1} \frac{1}{\sinh(d(x,z))} \, d x + \int_{D_2} \frac{1}{\sinh(d(x,z')) } \, d x,
    \end{aligned}
\end{equation}
where $B(z',t')$ is the ball centered at $z'$ with radius $t'$ and the sets $D_1$ and $D_2$ are given by
\begin{equation*}
    \begin{aligned}
        D_1 = B(z',t')\cap \big\{ x \in \bH \,\big| \, d(x,z) < d(x,z')   \big\} \qquad\text{and}\qquad D_2 = B(z',t')\cap \big\{ x \in \bH \,\big|  \,d(x,z) \geq d(x,z')   \big\}.
    \end{aligned}
\end{equation*}
For each of the two integrals on the right-hand side, we switch polar coordinates centred at $z$ and $z'$ respectively and obtain
\begin{equation}\label{EQ:Two_sqr_sinh_singu_3}
    \begin{aligned}
          \int_{D_1} \frac{1}{\sinh(d(x,z))} \, d x + \int_{D_2} \frac{1}{\sinh(d(x,z')) } \, d x
          \leq 4\pi \int_0^{t'} 1 dr = 4\pi t'
    \end{aligned}
\end{equation}

Suppose that $2t'<d(z,z') \leq t+t'$. For this case, we find that if $x\in B(z',t)$, then we have $d(x,z)\geq d(x,z')$. Hence, if we switch to polar coordinates centred at $z'$ we find that
\begin{equation}\label{EQ:Two_sqr_sinh_singu_4}
    \begin{aligned}
        \int_{\bH} \frac{\boldsymbol{1}(d(x,z)<t)\boldsymbol{1}(d(x,z')<t')}{\sqrt{\sinh(d(x,z)) }\sqrt{\sinh(d(x,z')) }} \, d x  \leq{}&  \int_0^{2\pi} \int_0^{t'} \frac{\sinh(r)}{\sqrt{\sinh(R(r,\theta)) \sinh(r)}} \, d r d\theta 
        \\
        \leq{}& 2\pi   \int_0^{t'} 1 \, d r = 2\pi t', 
    \end{aligned}
\end{equation}
where $R(r,\theta)$ is $d(x,z)$ written in these polar coordinates. Combining the estimates in \eqref{EQ:Two_sqr_sinh_singu_1}, \eqref{EQ:Two_sqr_sinh_singu_2}, \eqref{EQ:Two_sqr_sinh_singu_3}, and \eqref{EQ:Two_sqr_sinh_singu_4} we obtain the state estimate, and this concludes the proof. 
\end{proof}

\subsection{Restriction to thick parts of the surface} 

Here, we establish that it will be sufficient to consider the propagated observable on the thick part of the surface. Recall the notation
\begin{equation*}
    X(> 2T) \coloneqq \{ x\in X \,|\, \InjRad_X(x) \geq 2T \} \qquad\text{and}\qquad X(\leq 2T) = X\setminus X(> 2T),
\end{equation*}
for all $T>0$. Moreover, we will use similar notation for the fundamental domain of $X$. Before we prove the main result of this section, we will need an auxiliary lemma on Hilbert-Schmidt norm bounds of free propagation. 

\begin{lemma}\label{LE:free_propagation_restriction_1}
    Let $X=\Gamma\setminus \bH$, where $\Gamma$ be a strictly hyperbolic Fuchsian group, and let $\tau\in\R$. Let $P_0(t)$ and $\widetilde{P}_0(t,\tau)$ be the free and modified free wave propagator associated with $-\Delta_X$. Then for all $a\in L^\infty(X)$, $0<r\leq \min(1, \frac{\InjRad_X}{2})$ and all $T>0$ we have for any $t,t'\in \R^+$ such that $\max(t,t')\leq T$ the estimates
    \begin{equation*}
        \big\lVert  P_0(t) a \boldsymbol{1}_{X(\leq2T)} P_0(t')   \big\rVert_{\bHS}^2 \leq \frac{8\pi T e^{4T+1} \lVert a \rVert^2_{L^\infty(X)}}{r^2} \Vol(X(\leq2T)),
    \end{equation*}
    and
    \begin{equation*}
        \big\lVert  \widetilde{P}_0(t,\tau) a \boldsymbol{1}_{X(\leq2T)} P_0(t')   \big\rVert_{\bHS}^2 \leq \frac{8\pi T e^{4T+1} \lVert a \rVert^2_{L^\infty(X)}}{r^2} \Vol(X(\leq2T)).
    \end{equation*}
\end{lemma}

\begin{proof}
Firstly, we notice that since $\widetilde{P}_0(t,\tau) = \cos(\tau t) P_0(t)$ we have
\begin{equation*}
     \big\lVert  \widetilde{P}_0(t,\tau) a \boldsymbol{1}_{X(\leq2T)} P_0(t')   \big\rVert_{\bHS}^2 \leq  \big\lVert  P_0(t) a \boldsymbol{1}_{X(\leq2T)} P_0(t')   \big\rVert_{\bHS}^2.
\end{equation*}
Hence, we only need to establish the bound for $ \big\lVert  P_0(t) a \boldsymbol{1}_{X(\leq2T)} P_0(t')   \big\rVert_{\bHS}^2$. Using that $P_0(t)$ has an explicit kernel that is positive, we obtain
\begin{equation}\label{EQ:free_propagation_restriction_1_1}
    \begin{aligned}
         \big\lVert  P_0(t) a \boldsymbol{1}_{X(\leq2T)} P_0(t')   \big\rVert_{\bHS}^2 
         ={}& \int_D \int_D \Big| \sum_{\gamma_1,\gamma_2\in\Gamma} \int_{X(\leq2T)} A(t,d(x,\gamma_1w)) a(w) A(t',d(y,\gamma_2 w)) \, d w  \Big|^2 \, d xd y
         \\
         \leq{}& \norm{a}^2_{L^\infty(X)} \int_{D^3}  \int_{D(\leq2T)} \sum_{\gamma_1,\gamma_2,\gamma_3,\gamma_4\in\Gamma}  A(t,d(x,\gamma_1w)) A(t',d(y,\gamma_2 w))
         \\
         &\phantom{\norm{a}^2_{L^\infty(X)} }{} \times  A(t,d(x,\gamma_3w')) A(t',d(y,\gamma_4 w')) \, d w  d w' d xd y,
    \end{aligned}
\end{equation}  
where we could also have restricted the integration domain for $w'$ to $D(\leq2T)$. Using that the fundamental domain tiles the hyperbolic plane we obtain
\begin{equation}\label{EQ:free_propagation_restriction_1_2}
    \begin{aligned}
          \MoveEqLeft \int_{D^3}  \int_{D(\leq2T)}  \sum_{\gamma_1,\gamma_2,\gamma_3,\gamma_4\in\Gamma}  A(t,d(x,\gamma_1w)) A(t',d(y,\gamma_2 w))
          A(t,d(x,\gamma_3w')) A(t',d(y,\gamma_4 w')) \, d w  d w' d xd y
          \\
          ={}& \int_{\bH^3}  \int_{D(\leq2T)}  \sum_{\gamma \in\Gamma}  A(t,d(x,w)) A(t',d(y,\gamma w))
          A(t,d(x,w')) A(t',d(y,w')) \, d w  d w' d xd y.
    \end{aligned}
\end{equation}
Now, applying Lemma~\ref{LE:Two_abel_kernel_int_est} to the integrals in $x$ and $y$ we obtain the following estimate 
\begin{equation}\label{EQ:free_propagation_restriction_1_3}
    \begin{aligned}
          \MoveEqLeft \int_{\bH^3}  \int_{D(\leq2T)}  \sum_{\gamma \in\Gamma}  A(t,d(x,w)) A(t',d(y,\gamma w))
          A(t,d(x,w')) A(t',d(y,w')) \, d w  d w' d xd y
          \\
          \leq{}& \int_{\bH}  \int_{D(\leq2T)}  \sum_{\gamma \in\Gamma} \frac{ \boldsymbol{1}_{[0,t+t']}(d(\gamma w, w'))\boldsymbol{1}_{[0,t+t']}(d(w,w')) }{\sqrt{\sinh (d(w,\gamma w')) \sinh (d(w,w'))}}   \, d w  d w'. 
    \end{aligned}
\end{equation}
Then applying Lemma~\ref{LE:Two_sqr_sinh_singu_1} to the integral in $w'$ and Lemma~\ref{LE:counting_loop_lemma} to control the sum over $\Gamma$ we obtain  
\begin{equation}\label{EQ:free_propagation_restriction_1_4}
          \int_{\bH}  \int_{D(\leq2T)}  \sum_{\gamma \in\Gamma} \frac{ \boldsymbol{1}_{[0,t+t']}(d(\gamma w, w'))\boldsymbol{1}_{[0,t+t']}(d(w,w')) }{\sqrt{\sinh (d(\gamma w, w')) \sinh (d(w,w'))}}   \, d w  d w' 
          \leq   \frac{4 \pi e^{2(t+t')+1 }(t+t')}{ r^2} \Vol(  D(\leq2T)).
\end{equation}
By combining the estimates in \eqref{EQ:free_propagation_restriction_1_1}-\eqref{EQ:free_propagation_restriction_1_4}, using the assumption that $\max(t,t')\leq T$, and $ \Vol(D(\leq2T)) =  \Vol(  X(\leq2T))$, we obtain the stated estimate. This concludes the proof.
\end{proof}

\begin{lemma}\label{LE:lifting_type_lemma}
    Let $X=\Gamma\setminus \bH$, where $\Gamma$ be a strictly hyperbolic Fuchsian group, and let $\tau\in\R$. For a $V\in L^\infty(X)$ consider the Schr\"odinger operator $H_V = -\Delta_X + V$ and for $t\geq 0$ the associated wave and modified wave propagator $P_V(t)$ and $\widetilde{P}_V(t,\tau)$. Moreover, suppose that $I\subset (\frac{1}{4},\infty)$ is a fixed compact interval. Then for all $a\in L^\infty(X)$, $0<r\leq \min(1, \frac{\InjRad_X}{2})$ and all $T\geq1$ we have
    \begin{equation*}
    \begin{aligned}
       \MoveEqLeft \Big\lVert \int_0^T \Pi_I(H_V) P_V(t) a P_V(t) \Pi_I(H_V) \, d t \Big\rVert_{\bHS}^2 
       \\
       \leq{}& 2\Big\lVert \int_0^T \Pi_I(H_V) P_V(t) a \boldsymbol{1}_{X(>2T)} P_V(t) \Pi_I(H_V) \, d t  \Big\rVert_{\bHS}^2 
        \\
        &+ \frac{C_I T^7 e^{4T+1}\max(1, \lVert V\rVert_{L^\infty(X)}^4) \lVert a \rVert^2_{L^\infty(X)}}{r^2} \Vol(X(\leq2T)), 
        \end{aligned}
    \end{equation*}
    and 
    \begin{equation*}
    \begin{aligned}
       \MoveEqLeft \Big\lVert \int_0^T \Pi_I(H_V) \widetilde{P}_V(t,\tau) a P_V(t) \Pi_I(H_V) \, d t \Big\rVert_{\bHS}^2 
       \\
       \leq{}& 2\Big\lVert \int_0^T \Pi_I(H_V) \widetilde{P}_V(t,\tau) a \boldsymbol{1}_{X(>2T)} P_V(t) \Pi_I(H_V) \, d t  \Big\rVert_{\bHS}^2 
        \\
        &+ \frac{C_I T^7 e^{4T+1}\max(1, \lVert V\rVert_{L^\infty(X)}^4) \lVert a \rVert^2_{L^\infty(X)}}{r^2} \Vol(X(\leq2T)). 
        \end{aligned}
    \end{equation*}
    In both cases, the constant $C_I = 16 \pi (\max(1,(\min(I)-\frac14)^{-1})$.
\end{lemma}

\begin{proof}
 The proof of the two bounds are analogous and we will only give a proof for the first bound. Since the sets $X(>2T)$ and $X(\leq 2T)$ are disjoint, we have
        \begin{equation}\label{EQ:lifting_type_lemma_1}
        \begin{aligned}
            \Big\lVert \int_0^T \Pi_I(H_V)P_V(t) a P_V(t)\Pi_I(H_V) \,  d t \Big\rVert_{\bHS}^2 
            \leq{}& 2\Big\lVert \int_0^T \Pi_I(H_V)P_V(t) a \boldsymbol{1}_{X(>2T)} P_V(t) \Pi_I(H_V)\,  d t  \Big\rVert_{\bHS}^2 
            \\
            &+ 2\Big\lVert \int_0^T \Pi_I(H_V)P_V(t) a \boldsymbol{1}_{X(\leq2T)} P_V(t) \Pi_I(H_V)\, d t  \Big\rVert_{\bHS}^2.  
        \end{aligned}
    \end{equation}
    Hence, we just need to bound the second term on the right-hand side. In the following, we let $\tilde{a} = a \boldsymbol{1}_{X(\leq2T)}$. Using Duhamel's formula \eqref{Duhamel_formula_main} for $P_V(t)$, we have
    \begin{equation}\label{EQ:lifting_type_lemma_2}
    \begin{aligned}
        \MoveEqLeft \Big\lVert \int_0^T \Pi_I(H_V)P_V(t) \tilde{a} P_V(t)\Pi_I(H_V) \,  d t  \Big\rVert_{\bHS}^2 
        \\
        ={}& \Big\lVert \int_0^T \Pi_I(H_V)(P_0(t) -Q_V(t)) \tilde{a} (P_0(t) -Q_V^{*}(t))\Pi_I(H_V) \,  d t  \Big\rVert_{\bHS}^2
        \\
        \leq{}& \Bigg( \Big\lVert \int_0^T \Pi_I(H_V)P_0(t)  \tilde{a} P_0(t) \Pi_I(H_V) \,  d t  \Big\rVert_{\bHS}  + \Big\lVert \int_0^T \Pi_I(H_V) P_0(t)  \tilde{a}  Q_V^{*}(t)\Pi_I(H_V) \,  d t  \Big\rVert_{\bHS}
        \\
        &+\Big\lVert \int_0^T \Pi_I(H_V)Q_V(t) \tilde{a} P_0(t) \Pi_I(H_V) \,  d t  \Big\rVert_{\bHS} + \Big\lVert \int_0^T \Pi_I(H_V)Q_V(t) \tilde{a} Q_V^{*}(t)\Pi_I(H_V) \,  d t  \Big\rVert_{\bHS} \Bigg)^2.
    \end{aligned} 
    \end{equation}
    We will now estimate each of the terms on the right-hand side separately. For the first term, using Lemma~\ref{LE:free_propagation_restriction_1} we obtain the following estimate
    \begin{equation}\label{EQ:lifting_type_lemma_3}
    \begin{aligned}
         \Big\lVert \int_0^T \Pi_I(H_V)P_0(t) \tilde{a} P_0(t) \Pi_I(H_V)\, d t \Big\rVert_{\bHS} 
        \leq{}&  \int_0^T \big\lVert P_0(t) \tilde{a} P_0(t) \big\rVert_{\bHS} \, d t 
        \\
        \leq{}& \frac{ \sqrt{8\pi} T^{\frac{3}{2}} e^{2T+\frac{1}{2}} \lVert a \rVert_{L^\infty(X)}}{r} \sqrt{\Vol(X(\leq2T))}.
    \end{aligned} 
    \end{equation}
Recall that
    \begin{equation*}
    Q_V(t) =  \int_{0}^{t}  P_V(t_1) V  P_0(t-t_1) \, dt_1,
    \end{equation*}
    we obtain for the second term the estimate
    \begin{equation}\label{EQ:lifting_type_lemma_4}
    \begin{aligned}
        \MoveEqLeft  \Big\lVert \int_0^T \Pi_I(H_V) P_0(t)  \tilde{a}  Q_V^{*}(t)\Pi_I(H_V) \,  d t  \Big\rVert_{\bHS}
         \\
         &\leq \tilde{C}_I \norm{V}_{L^\infty(X)} \int_0^T \int_0^t  \big\lVert P_0(t) a \boldsymbol{1}_{X(\leq2T)}  P_0(t-s) \big\rVert_{\bHS} \,  d s  d t 
        \\
        &\leq \tilde{C}_I \norm{V}_{L^\infty(X)}  \frac{ \sqrt{2\pi} T^{\frac{5}{2}} e^{2T+\frac12} \lVert a \rVert_{L^\infty(X)}}{r} \sqrt{\Vol(X(\leq2T))},
    \end{aligned} 
    \end{equation}
    where the constant $\tilde{C}_I = (\min(I)-\frac14)^{-1/2}$ and we have used Lemma~\ref{LE:free_propagation_restriction_1}. Notice that the operators in the second and third terms are adjoint to each other. Hence, the bound just obtained for the second term is also a bound for the third term. For the fourth term, we again use the definition of $Q_V(t)$ and Lemma~\ref{LE:free_propagation_restriction_1} to obtain the bound   
    \begin{equation}\label{EQ:lifting_type_lemma_5}
    \begin{aligned}
        \MoveEqLeft \Big\lVert \int_0^T \Pi_I(H_V)Q_V(t) \tilde{a} Q_V^{*}(t)\Pi_I(H_V) \,  d t  \Big\rVert_{\bHS} 
        \\
        &\leq \tilde{C}_I^2 \norm{V}_{L^\infty(X)}^2 \frac{ \sqrt{8\pi} T^{\frac{7}{2}} e^{2T+\frac12} \lVert a \rVert_{L^\infty(X)}}{3r} \sqrt{\Vol(X(\leq2T))},
    \end{aligned} 
    \end{equation}
    where again $\tilde{C}_I = (\min(I)-\frac14)^{-1/2}$. By combining the estimates in \eqref{EQ:lifting_type_lemma_2}-\eqref{EQ:lifting_type_lemma_5}, we obtain the inequality  
    \begin{equation}\label{EQ:lifting_type_lemma_6}
    \begin{aligned}
        \Big\lVert \int_0^T \Pi_I(H_V)P_V(t) \tilde{a} P_V(t)\Pi_I(H_V) \,  d t  \Big\rVert_{\bHS}^2 
        \leq C_I \frac{ T^7 e^{4T+1}\max(1, \lVert V\rVert_{L^\infty(X)}^4) \lVert a \rVert^2_{L^\infty(X)}}{r^2} \Vol(X(\leq2T)),
    \end{aligned} 
    \end{equation}
    where the constant $C_I = 16 \pi (\max(1,\tilde{C}_I^2)$. Finally, combining the estimates in \eqref{EQ:lifting_type_lemma_1} and \eqref{EQ:lifting_type_lemma_6}, we obtain the stated estimate, and this concludes the proof.
\end{proof}

\subsection{Bounds on terms involving the potential}

In this subsection, we will state some quantitative bounds that will be used in the proof of the main result. First, we state an auxiliary bound. The second bound is the main estimate used to control the error terms.

\begin{lemma}\label{LE:Potential_restriction_1}
Let $X=\Gamma\setminus \bH$, where $\Gamma$ is a strictly hyperbolic Fuchsian group, and let $\tau\in\R$. Suppose $V\in L^\infty(X)$ and let $P_0(t)$ and $\widetilde{P}_0(t,\tau)$ be the free and modified free wave propagator associated with $-\Delta_X$. Then for all $T\geq 1$ and $a\in L^\infty(X)$ we have for any $t,t'\in \R_{+}$ such that $\max(t,t')\leq T$ the estimates
\begin{equation*}
       \big\lVert  P_0(t) a \boldsymbol{1}_{X( > 2T)} P_0(t') V  \big\rVert_{\bHS}^2 \leq  4\pi (t+t') \norm{a}_{L^\infty(X)}^2 \norm{V}_{L^2(X)}^2
\end{equation*}
and
\begin{equation*}
       \big\lVert  \widetilde{P}_0(t,\tau) a \boldsymbol{1}_{X( > 2T)} P_0(t') V  \big\rVert_{\bHS}^2 \leq  4\pi (t+t') \norm{a}_{L^\infty(X)}^2 \norm{V}_{L^2(X)}^2
\end{equation*}
\end{lemma}

\begin{proof}
As in the proof of Lemma~\ref{LE:free_propagation_restriction_1}, we have 
\begin{equation*}
     \big\lVert  \widetilde{P}_0(t,\tau) a \boldsymbol{1}_{X( > 2T)} P_0(t') V  \big\rVert_{\bHS}^2 \leq  \big\lVert  P_0(t) a \boldsymbol{1}_{X( > 2T)} P_0(t') V  \big\rVert_{\bHS}^2
\end{equation*}
since $\widetilde{P}_0(t,\tau) = \cos(\tau t) P_0(t)$. Hence, we will only prove the first bound.
Letting $K_{t}^\Gamma(x,y)$ be the integral kernel of $P_0(t)$, we have
\begin{equation}\label{Eq:K_Gamma_1}
    \begin{aligned}
    \MoveEqLeft \big\lVert  P_0(t) a \boldsymbol{1}_{X( > 2T)} P_0(t') V  \big\rVert_{\bHS}^2 
    \\
    &= \int_D \int_D \big| \int_{D( > 2T)} K_t^\Gamma(x,w) a(w) K_{t'}^\Gamma(w,y)V(y) \, d w  \big|^2 \, d yd x
    \\
    &\leq \int_D \int_D \big| \int_{D( > 2T)} \sum_{ \gamma_2 \in \Gamma} \sum_{\gamma_1 \in \Gamma } A(t, d(x, \gamma_1 w) a(w) A(t', d( w, \gamma_2 y ))V(y) \, d w  \big|^2 \, d yd x,
    \end{aligned}
\end{equation}
where we have used the form of the kernel $K^\Gamma_{t}(x,y)$ as given in Lemma~\ref{LE:kernel_of_unper_pro}.  
Since the Abel kernel is positive, we see from taking the absolute value under the integral and performing  $ \gamma_1 \mapsto \gamma_1^{-1} \gamma_2^{-1} $ that  
\begin{equation}\label{Eq:K_Gamma_2}
    \begin{aligned}
    \MoveEqLeft \int_D \int_D \big| \int_{D( > 2T)} \sum_{ \gamma_2 \in \Gamma} \sum_{\gamma_1 \in \Gamma } A(t, d(x, \gamma_1 w) a(w) A(t', d( w, \gamma_2 y ))V(y) \, d w  \big|^2 \, d yd x
    \\
    \leq{}& \lVert a\rVert_\infty^2 \int_D \int_D |V(y) |^2 \big| \int_{D( > 2T)} \sum_{ \gamma_2 \in \Gamma} \sum_{\gamma_1 \in \Gamma } A(t, d( \gamma_1 x, \gamma_2^{-1} w) A(t', d( \gamma_2^{-1} w, y )) \, d w  \big|^2 \, d yd x
    \\
    ={}&\lVert a\rVert_\infty^2 \int_D \int_D |V(y) |^2 \big| \int_{\mathbb{H}( > 2T)} \sum_{\gamma \in \Gamma } A(t, d( \gamma x, w) A(t', d( w, y )) \, d w  \big|^2 \, d yd x,
    \end{aligned}
\end{equation}
where $ \bH(>2T) = \bigcup_{\gamma \in \Gamma} \gamma D(>2T) $. We then obtain that 
\begin{equation}\label{Eq:K_Gamma_3}
    \begin{aligned}
    \MoveEqLeft \int_D \int_D |V(y) |^2 \big| \int_{\mathbb{H}( > 2T)} \sum_{\gamma \in \Gamma } A(t, d( \gamma x, w) A(t', d( w, y )) \, d w  \big|^2 \, d yd x
    \\
    ={}&\int_D  \int_D |V(y) |^2 \int_{[\mathbb{H}( > 2T)]^2} \sum_{\gamma, \tilde{\gamma} \in \Gamma } A(t, d( \gamma x, w) A(t', d( w, y )) A(t, d( \gamma x, \tilde{\gamma} \tilde{w}) A(t', d( \tilde{w}, y )) \, d \tilde{w} d w  d yd x
    \\
    ={}&\int_{\mathbb{H}} \int_D |V(y) |^2  \int_{[\mathbb{H}( > 2T)]^2}  \sum_{\gamma \in \Gamma } A(t, d(x, w) A(t', d( w, y )) A(t, d(x, \gamma \tilde{w}) A(t', d( \tilde{w}, y )) \, d \tilde{w} d w  d yd x
    \\
    ={}&\int_{\mathbb{H}} \int_D |V(y) |^2  \int_{[\mathbb{H}( > 2T)]^2}   A(t, d(x, w) A(t', d( w, y )) A(t, d(x, \tilde{w}) A(t', d( \tilde{w}, y )) \, d \tilde{w} d w  d yd x,
    \end{aligned}
\end{equation}
where in the first equality, we have expanded the square and made the change of variables $\tilde{\gamma} \mapsto \tilde{\gamma}^{-1} \gamma $. Due to the properties of $A$, we have that $d(x,w), d(x, \gamma \tilde{w}) < t$ and $d(w,y), d(\tilde{w},y) < t'$. Finally in the last equality we have used that in the remaining sum, only the identity element can have a non-zero contribution. To see this, assume otherwise. We have that 
\begin{equation*}
    d( \tilde{w}, \gamma \tilde{w}) \leq d( \tilde{w}, x) + d(x, \gamma \tilde{w}) \leq d(\tilde{w},x) + T \leq d(\tilde{w}, w) + d(w,x) + T \leq 2T + T + T \leq 4T.
\end{equation*}
However, since $\tilde{w} \in \mathbb{H}(>2T)$, we have that
\begin{equation*}
    \frac{1}{2} \min_{ \gamma \in \Gamma \setminus\{ \mathrm{id} \} } d( \tilde{w}, \gamma \tilde{w} ) > 2T \Rightarrow \forall \gamma \in \Gamma \setminus \{\mathrm{id} \} : d(\tilde{w}, \gamma \tilde{w}) > 4T.
\end{equation*}
Then by applying Lemma~\ref{LE:Two_abel_kernel_int_est} to the integrals in $w$ and $\tilde{w}$, we obtain the following
\begin{equation}\label{Eq:K_Gamma_4}
    \begin{aligned}
    \MoveEqLeft \int_{\mathbb{H}} \int_D |V(y) |^2  \int_{[\mathbb{H}( > 2T)]^2}  \sum_{\gamma \in \Gamma } A(t, d(x, w) A(t', d( w, y )) A(t, d(x, \gamma \tilde{w}) A(t', d( \tilde{w}, y )) \, d \tilde{w} d w  d yd x
    \\
    \leq{}& \int_{\mathbb{H}} \int_D |V(y) |^2 \big| \int_{\mathbb{H}( > 2T)} \int_{\mathbb{H}( > 2T)}  A(t, d(x, w) A(t', d( w, y )) A(t, d(x, \tilde{w}) A(t', d( \tilde{w}, y )) \, d \tilde{w} d w  d yd x
    \\
    \leq{}& \int_{\bH^2}   \frac{ \boldsymbol{1}_{[0,t+t']}(d(x,y))}{\sinh (\max(|t-t'|,d(x, y)))}   |V(y)\boldsymbol{1}_D(y)|^2  \, dx dy
    \\
    &\leq 4\pi (t+t') \int_{\bH} |V(y)\boldsymbol{1}_D(y)|^2 \,dy
    \\
    &= 4\pi (t+t') \norm{V}_{L^2(X)}^2, 
    \end{aligned}
\end{equation}
where in the second inequality, we have used Lemma~\ref{LE:Two_sqr_sinh_singu_1} and the invariance over $\Gamma$ of the volume measure. Combining the estimates in \eqref{Eq:K_Gamma_1}, \eqref{Eq:K_Gamma_2}, \eqref{Eq:K_Gamma_3}, and \eqref{Eq:K_Gamma_4} we obtain the stated estimate, and this concludes the proof.
\end{proof}

\begin{lemma}\label{LE:Potential_restriction_2}
Let $X=\Gamma\setminus \bH$, where $\Gamma$ be a strictly hyperbolic Fuchsian group, and let $\tau\in\R$. For $V\in L^\infty(X)$ consider the Schr\"odinger operator $H_V = -\Delta_X + V$ and for $t\geq 0$ the associated wave and modified wave propagator $P_V(t)$ and $\widetilde{P}_V(t,\tau)$. Moreover, suppose that $I\subset (\frac{1}{4},\infty)$ is a fixed compact interval. Then for all $a\in L^\infty(X)$ and all $T\geq1$ we have
\begin{equation*}
    \begin{aligned}
        \MoveEqLeft \Big\lVert \int_0^T \Pi_I(H_V) P_V(t) a \boldsymbol{1}_{X(>2T)} P_V(t) \Pi_I(H_V) \, d t \Big\rVert_{\bHS}^2 
       \\
       &\leq 2\Big\lVert \int_0^T P_0(t) a \boldsymbol{1}_{X(>2T)} P_0(t)   \, d t \Big\rVert_{\bHS}^2 +  C_I T^{7} \norm{a}_{L^\infty(X)}^2 \max(1,\norm{V}_{L^\infty(X)}^2) \norm{V}_{L^2(X)}^2, 
    \end{aligned}
\end{equation*}
and
\begin{equation*}
    \begin{aligned}
        \MoveEqLeft \Big\lVert \int_0^T \Pi_I(H_V) \widetilde{P}_V(t,\tau) a \boldsymbol{1}_{X(>2T)} P_V(t) \Pi_I(H_V) \, d t \Big\rVert_{\bHS}^2 
       \\
       &\leq 2\Big\lVert \int_0^T \widetilde{P}_0(t,\tau) a \boldsymbol{1}_{X(>2T)} P_0(t)   \, d t \Big\rVert_{\bHS}^2 +  C_I T^{7} \norm{a}_{L^\infty(X)}^2 \max(1,\norm{V}_{L^\infty(X)}^2) \norm{V}_{L^2(X)}^2, 
    \end{aligned}
\end{equation*}
where the constant $C_I=\big(100 \sqrt{\pi }\max(1,(\min(I)-\frac14)^{-1})\big)^2$ in both cases.
\end{lemma}

\begin{proof}
The proof of both bounds are analogous, and we will only give a proof for the first bound. In the following, we let $\tilde{a} = a\boldsymbol{1}_{X(>2T)}$. Using the Duhamel formula \eqref{Duhamel_formula_main} we obtain the following estimate 
\begin{equation}\label{EQ:Potential_restriction_2_1}
    \begin{aligned}
       \MoveEqLeft \Big\lVert \int_0^T \Pi_I(H_V) P_V(t) \tilde{a} P_V(t) \Pi_I(H_V) \, d t \Big\rVert_{\bHS}^2 
       \\
       \leq{}& \Big\lVert \int_0^T \Pi_I(H_V) (P_0(t) - Q_V(t))\tilde{a} (P_0(t) -Q_V^*(t)) \Pi_I(H_V) \, d t \Big\rVert_{\bHS}^2 
       \\
       \leq{}& 2\Big\lVert \int_0^T \Pi_I(H_V) P_0(t)  \tilde{a} P_0(t)  \Pi_I(H_V) \, d t \Big\rVert_{\bHS}^2
       +
       2\Big(\Big\lVert \int_0^T \Pi_I(H_V) P_0(t)  \tilde{a} Q_V^*(t) \Pi_I(H_V) \, d t \Big\rVert_{\bHS} 
       \\
       &+
       \Big\lVert \int_0^T \Pi_I(H_V) Q_V(t) \tilde{a} P_0(t)  \Pi_I(H_V) \, d t \Big\rVert_{\bHS} 
       +
       \Big\lVert \int_0^T \Pi_I(H_V) Q_V(t))\tilde{a} Q_V^*(t) \Pi_I(H_V) \, d t \Big\rVert_{\bHS}\Big)^2 
    \end{aligned}
\end{equation}
We will estimate each of these terms separately. Firstly, since $\Pi_I(H_V)$ is a spectral projection, we have
\begin{equation}\label{EQ:Potential_restriction_2_2}
    \Big\lVert \int_0^T \Pi_I(H_V) P_0(t)  \tilde{a} P_0(t)  \Pi_I(H_V) \, d t \Big\rVert_{\bHS}^2
       \leq \Big\lVert \int_0^T P_0(t)  \tilde{a} P_0(t)   \, d t \Big\rVert_{\bHS}^2.
\end{equation}
Next, notice that since the Hilbert-Schmidt norm is invariant under conjugation, we see that the second and third norms are equal. Hence, we will only estimate the second term. To estimate this norm, recall that 
\begin{equation*}
    Q_V^*(t) =  \int_{0}^{t} P_0(t-t_1) V P_V(t_1) \, dt_1
\end{equation*}
we then obtain the following estimate 
\begin{equation*}
    \Big\lVert \int_0^T \Pi_I(H_V) P_0(t)  \tilde{a} Q_V^*(t) \Pi_I(H_V) \, d t \Big\rVert_{\bHS} 
       \leq{} \tilde{C}_I  \int_0^T \int_{0}^{t}  \big\lVert P_0(t)  \tilde{a} P_0(t-t_1) V  \big\rVert_{\bHS}   \, d t_1 d t,  
\end{equation*}
where the constant $\tilde{C}_I = (\min(I)-\frac14)^{-1/2}$. Now applying Lemma~\ref{LE:Potential_restriction_1}, we obtain
\begin{equation*}
    \begin{aligned}
     \int_0^T \int_{0}^{t}  \big\lVert P_0(t)  \tilde{a} P_0(t-t_1) V  \big\rVert_{\bHS}   \, d t_1 d t
     \leq{}&   \int_0^T \int_{0}^{t} 2 \sqrt{\pi(2t - t_1)} \norm{a}_{L^\infty(X)}\norm{V}_{L^2(X)}  \, d t_1 d t
     \\
     \leq {}&  14 \sqrt{\pi} \tilde{C}_I T^{\frac{3}{2}}  \norm{a}_{L^\infty(X)}\norm{V}_{L^2(X)} 
     \end{aligned}
\end{equation*}
Combining these two estimates, we get
\begin{equation}\label{EQ:Potential_restriction_2_3}
    \begin{aligned}
       \MoveEqLeft \Big\lVert \int_0^T \Pi_I(H_V) P_0(t)  \tilde{a} Q_V^*(t) \Pi_I(H_V) \, d t \Big\rVert_{\bHS} 
       +
       \Big\lVert \int_0^T \Pi_I(H_V) Q_V(t) \tilde{a} P_0(t)  \Pi_I(H_V) \, d t \Big\rVert_{\bHS} 
       \\
       &\leq  28\sqrt{\pi} \tilde{C}_I T^{\frac{3}{2}}  \norm{a}_{L^\infty(X)}\norm{V}_{L^2(X)}. 
    \end{aligned}
\end{equation}
For the last Hilbert-Schmidt norm we again use the expression of $Q_V(t)$ and obtain the estimate
\begin{equation*}
    \begin{aligned}
       \MoveEqLeft \Big\lVert \int_0^T \Pi_I(H_V) Q_V(t))\tilde{a} Q_V^*(t) \Pi_I(H_V) \, d t \Big\rVert_{\bHS}
      \\
      &\leq \tilde{C}_I^2 \int_0^T \int_{0}^t \int_{0}^t  \big\lVert  VP_0(t-t_1)\tilde{a} P_0(t-t_2) V   \Big\rVert_{\bHS} \, d t_1  d t_2  d t 
      \\
      &\leq \tilde{C}_I^2 \norm{V}_{L^\infty(X)} \int_0^T \int_{0}^t \int_{0}^t  \big\lVert P_0(t-t_1)\tilde{a}P_0(t-t_2) V   \Big\rVert_{\bHS} \, d t_1  d t_2  d t, 
    \end{aligned}
\end{equation*}
where the constant $\tilde{C}_I$ is the same as above. Applying Lemma~\ref{LE:Potential_restriction_1}, we obtain 
\begin{equation*}
    \begin{aligned}
        \MoveEqLeft \int_0^T \int_{0}^t \int_{0}^t  \big\lVert P_0(t-t_1)\tilde{a}P_0(t-t_2) V   \Big\rVert_{\bHS} \, d t_1  d t_2  d t  
        \\
       &\leq \int_0^T \int_{0}^t \int_{0}^t 2 \sqrt{\pi(2t - t_1-t_2)} \norm{a}_{L^\infty(X)}\norm{V}_{L^\infty(X)} e^{\frac{R_V}{2}}  \, d t_1  d t_2  d t
       \\
       &\leq 100\sqrt{\pi} T^{\frac{7}{2}} \norm{a}_{L^\infty(X)}\norm{V}_{L^\infty(X)}.
    \end{aligned}
\end{equation*}
Combining these two estimates, we obtain
\begin{equation}\label{EQ:Potential_restriction_2_4}
       \Big\lVert \int_0^T \Pi_I(H_V) Q_V(t))\tilde{a} Q_V^*(t) \Pi_I(H_V) \, d t \Big\rVert_{\bHS}
      \leq  100\sqrt{\pi} \tilde{C}_I^2 T^{\frac{7}{2}} \norm{a}_{L^\infty(X)}\norm{V}_{L^\infty(X)} \norm{V}_{L^2(X)}.
\end{equation}
Combining the estimates in \eqref{EQ:Potential_restriction_2_1}-\eqref{EQ:Potential_restriction_2_4} we obtain the stated estimate, with the constant $C_I$ as stated in the lemma. This concludes the proof.
\end{proof}

\subsection{Estimates for free terms}

The proof of the estimates in this subsection will be based on the methods introduced in \cite{Hippi}. We cannot use the main estimates from \cite{Hippi} directly since we have here done the lifting to the universal cover in a slightly different way. However, the main ideas are the same in the two proofs, and they will also overlap. In the following, we will define the function  $F(t,t',d(w,w'))$ by
\begin{equation}\label{def_of_F}
   F(t,t',d(w,w')) = \int_{\bH} A(t,d(x,w))   A(t',d(x,w'))  \,d x \int_{\bH}    A(t,d(y,w))   A(t',d(y,w'))  \,d y,
\end{equation} 
where $t,t'>0$ and $w,w'\in\bH$. The function $F$ a function of just the distance between $w$ and $w'$, as follows from \cite[Lemma 7.1]{Hippi}. For this function, we have the following integral estimate:
\begin{lemma}\label{LE:int_est_after_exp_mix}
    Let $\beta$ be a positive constant smaller than 1 and let $0<t'<t$. Then
    \begin{equation}\label{LEEQ:int_est_after_exp_mix_1}
       F(t,t',r)\leq \frac{\boldsymbol{1}_{[0,t+t'](r)}}{\sinh(\max(t-t',r))},
    \end{equation}
    and moreover this implies that for $\kappa>1$
    \begin{equation}\label{LEEQ:int_est_after_exp_mix_2}
        \int_{0}^{t+t'} \sinh(r) (1+r)^{-\kappa}   F(t,t',r)\, d r \leq \begin{cases} \frac{\kappa}{\kappa-1}(1+t-t')^{1-\kappa}, & 1 < \kappa < 2 \\ 
        2(1+t-t')^{-1},& \kappa \geq 2.
        \end{cases}
    \end{equation}
\end{lemma}

Note that the second upper bound in \eqref{LEEQ:int_est_after_exp_mix_2} is not optimal, but it will suffice for later estimates.

\begin{proof}
The bound \eqref{LEEQ:int_est_after_exp_mix_1} follows from the definition of the function $F(t,t',r)$ and Lemma~\ref{LE:Two_abel_kernel_int_est}.
For the estimate in \eqref{LEEQ:int_est_after_exp_mix_2} consider the two cases $1<\kappa<2$ and $\kappa\geq 2$. 
Firstly, suppose that $1<\kappa<2$. We then have
    \begin{equation}\label{EQ:int_est_after_exp_mix_1}
    \begin{aligned}
       \int_{0}^{t+t'} \sinh(r) (1+r)^{-\kappa}   F(t,t',r)\, d r \leq {}&  \frac{1}{\sinh(t-t')}\int_{0}^{t-t'} \sinh(r) (1+r)^{-\kappa}    \, d r 
       + \int_{t-t'}^{t+t'}  (1+r)^{-\kappa}   \, d r  
       \end{aligned}
    \end{equation}
Using that $\sinh(r) (1+r)^{-\kappa}$ is an increasing function in $r$ we obtain for the first integral on the right hand side in \eqref{EQ:int_est_after_exp_mix_1} that
\begin{equation}\label{EQ:int_est_after_exp_mix_2}
    \begin{aligned}
         \frac{1}{\sinh(t-t')}\int_{0}^{t-t'} \sinh(r) (1+r)^{-\kappa}    \, d r  \leq (1+t-t')^{-\kappa} \int_{0}^{t-t'} 1   \, d r \leq (1+t-t')^{1-\kappa}.
    \end{aligned}
\end{equation}
For the second integral on the right hand side in \eqref{EQ:int_est_after_exp_mix_1} we get by evaluating the integral that
    \begin{equation}\label{EQ:int_est_after_exp_mix_3}
    \begin{aligned}
       \int_{t-t'}^{t+t'}  (1+r)^{-\kappa}   \, d r = \frac{1}{\kappa-1} (1+t-t')^{1-\kappa} - \frac{1}{\kappa-1} (1+t+t')^{1-\kappa}.   
       \end{aligned}
    \end{equation}
Combing the estimates in \eqref{EQ:int_est_after_exp_mix_1}, \eqref{EQ:int_est_after_exp_mix_2} and \eqref{EQ:int_est_after_exp_mix_3} we obtain the estimate
\begin{equation}\label{EQ:int_est_after_exp_mix_4}
    \begin{aligned}
       \int_{0}^{t+t'} \sinh(r) (1+r)^{-\kappa}   F(t,t',r)\, d r \leq {}&  \frac{\kappa}{\kappa -1}(1+t-t')^{1-\kappa}.  
    \end{aligned}
\end{equation}
For the case $\kappa\geq 2$ we notice that $(1+r)^{-\kappa} \leq (1+r)^{-2}$ and that $\sinh(r) (1+r)^{-2}$ is also an increasing function in $r$. Then by arguing as in the previous case we obtain
\begin{equation}
    \begin{aligned}
       \int_{0}^{t+t'} \sinh(r) (1+r)^{-\kappa}   F(t,t',r)\, d r \leq {}&  2(1+t-t')^{-1}.
    \end{aligned}
\end{equation}
This concludes the proof.
\end{proof}

With this we are now ready to state and prove the main estimate of this subsection.

\begin{lemma}\label{LE:free_HS_est_exponential_mixing}
Let $X=\Gamma\setminus \bH$, where $\Gamma$ be a strictly hyperbolic Fuchsian group, and let $\tau\in\R$. Let $P_0(t)$ and $\widetilde{P}_0(t,\tau)$ be the free and modified free wave propagator associated with $-\Delta_X$. Moreover, suppose that $a\in L^\infty(X)$ with zero average and that there exist $C_1>0$ and $\kappa>1$ such that  
\begin{equation}\label{LEEQ:free_HS_est_exponential_mixing_1}
  \big|\langle a \circ \phi_r, a \rangle \big| \leq C_1 (t+1)^{-\kappa} \lVert a \rVert_{L^2(X)}^2,\quad r \geq 0.  
\end{equation}
Then for all $T\geq 1$, we have the estimates
\begin{equation*}
       \Big\lVert \int_0^T  P_0(t) a \boldsymbol{1}_{X( > 2T)} P_0(t) \,d t  \Big\rVert_{\bHS}^2 \leq g_\kappa(T) \norm{a}^2_{L^2(X)}+  4\pi T^3 \norm{a}_{L^\infty(X)}^2 \vol(X(\leq 2T)),
\end{equation*}
and 
\begin{equation*}
       \Big\lVert \int_0^T  \widetilde{P}_0(t,\tau) a \boldsymbol{1}_{X( > 2T)} P_0(t) \,d t  \Big\rVert_{\bHS}^2 \leq g_\kappa(T) \norm{a}^2_{L^2(X)}+  4\pi T^3 \norm{a}_{L^\infty(X)}^2 \vol(X(\leq 2T)),
\end{equation*}
where 
\begin{equation*}
    g_\kappa(T) = \begin{cases}
        \frac{\kappa C_1}{\kappa -1} T^{3-\kappa} & \text{if $1<\kappa<2$,}\\
        2 C_1(T+1)\ln (T+1) & \text{if $\kappa\geq 2$.}
    \end{cases}        
\end{equation*}
\end{lemma}

\begin{proof}
We will here give the proof for the second estimate. The first is obtained by an analogous argument that omits the factor $\cos(\tau t)$. Letting $K_{t}^\Gamma(x,y)$ be the integral kernel of $P_0(t)$, we have
\begin{equation*}
   \Big\lVert \int_0^T  \widetilde{P}_0(t,\tau) a \boldsymbol{1}_{X( > 2T)} P_0(t) \,d t  \Big\rVert_{\bHS}^2 = \int_D \int_D \big|\int_0^T \cos(\tau t) \int_{D( > 2T)} K_t^\Gamma(x,w) a(w) K_{t}^\Gamma(w,y) \, d w d t  \big|^2 \, d yd x.
\end{equation*}
Recall that
\begin{equation*}
        K^\Gamma_{t}(x,y)= \sum_{\gamma \in \Gamma} A(t,d(x,\gamma y)).
\end{equation*}
By arguing as in the proof of Lemma~\ref{LE:Potential_restriction_1} we obtain after expanding the square the following identity
\begin{equation*}
    \begin{aligned}
         \MoveEqLeft \Big\lVert \int_0^T  P_0(t) a \boldsymbol{1}_{X( > 2T)} P_0(t) \,d t  \Big\rVert_{\bHS}^2 
         \\
         ={}&  \Big| \int_D \int_D \big| \int_0^T \cos(\tau t) \int_{D( > 2T)} \sum_{\gamma_1,\gamma_2\in\Gamma} A(t,d(\gamma_1 x,w)) a(w) A(t,d((w, \gamma_2 y)) \, d w  d t \big|^2 \, d yd x \Big|
         \\
         ={}&\Big| \int_0^T\int_0^T \cos(\tau t) \cos(\tau t') \int_{D( > 2T)}  \int_{B(w,2T)\cap \bH( > 2T) }  a(w) \overline{a(w')}  F(t,t',d(w,w'))d w' d w      d t'd t\Big|,
    \end{aligned}
\end{equation*}
where we view $a$ as a periodic function on $\bH$,  $ \bH(>2T) = \bigcup_{\gamma \in \Gamma} \gamma D(>2T) $, and the function $F(t,t',(d(w,w'))$ is as defined in \eqref{def_of_F}. Moreover, the restriction of the integration domain of $w'$ follows from the support properties of the Abel kernels. We then notice that due to the support properties of $F$ we have that 
\begin{equation*}
    \begin{aligned}
         \MoveEqLeft  \int_{D( > 2T)}  \int_{B(w,2T)\cap \bH( > 2T) }  a(w) \overline{a(w')}  F(t,t',d(w,w')) dx dyd w' d w  
         \\
         &=  \int_{D( > 2T)}  \int_{B(w,2T)\cap D( > 2T) }  a(w) \overline{a(w')} \sum_{\gamma \in\Gamma}  F(t,t',d(w, \gamma w')) dx dyd w' d w 
         \\
         &=\int_{D( > 2T)}  \int_{B(w,2T)\cap D( > 2T) }  a(w) \overline{a(w')}  F(t,t',d(w, w')) dx dyd w' d w, 
    \end{aligned}
\end{equation*}
where an argument analogous to the one used in the proof of Lemma~\ref{LE:Potential_restriction_1} gives us that only the identity element contributes. Using this we obtain that 
\begin{equation}\label{EQ:free_HS_est_exponential_mixing_1}
    \begin{aligned}
         \MoveEqLeft \Big\lVert \int_0^T  P_0(t) a \boldsymbol{1}_{X( > 2T)} P_0(t) \,d t  \Big\rVert_{\bHS}^2 
         \\
         ={}&\Big| \int_0^T\int_0^T \cos(\tau t) \cos(\tau t') \int_{D( > 2T)}  \int_{B(w,2T)\cap D( > 2T) }  a(w) \overline{a(w')}  F(t,t',d(w,w'))d w' d w      d t'd t\Big|,
    \end{aligned}
\end{equation}
By adding and subtracting the two terms, 
\begin{equation*}
     \int_0^T\int_0^T \cos(\tau t) \cos(\tau t') \int_{D( > 2T)}  \int_{B(w,2T)\cap D( \leq 2T) } a(w) \overline{a(w')} F(t,t',d(w,w'))  \,d w' d w      d t'd t
\end{equation*}
and
\begin{equation*}
     \int_0^T\int_0^T \cos(\tau t) \cos(\tau t') \int_{D( \leq 2T)}  \int_{B(w,2T)} a(w) \overline{a(w')}  F(t,t',d(w,w')) \,d w' d w      d t'd t.
\end{equation*}
We have that
\begin{equation}\label{EQ:free_HS_est_exponential_mixing_3}
    \begin{aligned}
         \MoveEqLeft \Big| \int_0^T\int_0^T \cos(\tau t) \cos(\tau t') \int_{D( > 2T)}  \int_{B(w,2T)\cap \bH( > 2T) } a(w) \overline{a(w')} F(t,t',d(w,w'))  \,d w' d w      d t'd t\Big|
         \\
         \leq {}&  \int_0^T\int_0^T \Big| \int_{D}  \int_{B(w,2T) } a(w) \overline{a(w')}  F(t,t',d(w,w')) \,d w' d w   \Big|   d t'd t
         \\
         &+ \int_0^T\int_0^T \Big| \int_{D( > 2T)}  \int_{B(w,2T)\cap D( \leq 2T) } a(w) \overline{a(w')} F(t,t',d(w,w'))  \,d w' d w    \Big|  d t'd t 
         \\
         &+ \int_0^T\int_0^T \Big| \int_{D( \leq 2T)}  \int_{B(w,2T)} a(w) \overline{a(w')}  F(t,t',d(w,w'))  \,d w' d w  \Big|    d t'd t 
         \\
         ={}& \mathcal{I}_1(a,T) +\mathcal{I}_2(a,T)+\mathcal{I}_3(a,T),
    \end{aligned}
\end{equation}
where we have used $|\cos(\tau t)|\leq 1$. We will first estimate $\mathcal{I}_2(a,T)$ and $\mathcal{I}_3(a,T)$. To do so, we notice that by applying Lemma~\ref{LE:Two_abel_kernel_int_est} to the integrals of the variables $x$ and $y$ in the definition of the function $ F(t,t',d(w,w'))$, we obtain the estimate 
\begin{equation}\label{EQ:free_HS_est_exponential_mixing_4}
    F(t,t',d(w,w')) \leq  \frac{\boldsymbol{1}_{[0,t'+t]}(d(w,w'))}{\sinh(\max(|t-t'|,d(w,w'))}.
\end{equation}
Moreover, since $ F(t,t',d(w,w'))$ is symmetric in $w$ and $w'$, we obtain the following estimate
\begin{equation}\label{EQ:free_HS_est_exponential_mixing_5}
    \begin{aligned}
        \mathcal{I}_2(a,T)+\mathcal{I}_3(a,T)
        &\leq 
         2\norm{a}^2_{L^\infty(X)}\int_0^T\int_0^T \int_{D} \int_{D( \leq 2T)} \frac{\boldsymbol{1}_{[0,t'+t]}(d(w,w'))}{\sinh(\max(|t-t'|,d(w,w'))}   \, d w  d w'   d t'd t 
         \\
         &\leq 4\pi T^3\norm{a}^2_{L^\infty(X)} \vol(X(\leq 2T)),
    \end{aligned}
\end{equation}
where we first evaluated the integral in $w'$ by switching to polar coordinates centered at $w$ and then the remaining integrals were evaluated. 
To estimate $\mathcal{I}_1(a,T)$ we first notice that $F(t,t',d(w,w'))$ is symmetric in $t$ and $t'$. Therefore, we obtain the following
\begin{equation*}
    \begin{aligned}
        \mathcal{I}_1(a,T)
        \leq 2  \int_0^T\int_0^t \Big| \int_{D}  \int_{B(w,2T) } a(w) \overline{a(w')}  F(t,t',d(w,w'))  \,d w' d w   \Big|   d t'd t
    \end{aligned}
\end{equation*}
Next we write $w'$ in polar coordinates around $w$ such that $w' = \pi_1(\varphi_r(w,\theta))$, this gives us 
\begin{equation*}
    \begin{aligned}
        \mathcal{I}_1(a,T)
        &\leq 2  \int_0^T\int_0^t \Big| \int_{D} \int_{0}^{2T}  \int_{\mathbb{S}^1 } a(w) \overline{a(\pi_1(\varphi_r(w,\theta))}  \sinh(r) F(t,t',r)  \,d \theta d r d w   \Big|   d t'd t
        \\
        &\leq 2 \int_0^T\int_0^t \int_{0}^{2T} \sinh(r)F(t,t',r)  \Big| \int_{D}   \int_{\mathbb{S}^1 } a(w) \overline{a(\pi_1(\varphi_r(w,\theta))}  \,d \theta d w   \Big| d r   d t'd t
        \\
        &= 2 \int_0^T\int_0^t \int_{0}^{2T} \sinh(r)F(t,t',r)  \big| \langle a, a\circ \varphi_r\rangle_{L^2(D\times \mathbb{S}^1)}   \big| d r   d t'd t.
    \end{aligned}
\end{equation*}
From our assumptions \eqref{LEEQ:free_HS_est_exponential_mixing_1} we have that there exist a $C_1>0$ and $\kappa>1$ such that 
\begin{equation*}
  \big|\langle a \circ \phi_r, a \rangle \big| \leq C_1 (t+1)^{-\kappa} \lVert a \rVert_{L^2(X)}^2,\quad r \geq 0.  
\end{equation*}
Hence, using Lemma~\ref{LE:int_est_after_exp_mix} we obtain the following estimate
\begin{equation}\label{EQ:free_HS_est_exponential_mixing_7}
    \begin{aligned}
        \mathcal{I}_1(a,T)
        &\leq C_1 \max\big(\tfrac{\kappa}{\kappa-1},2\big)\norm{a}^2_{L^2(X)} \int_0^T\int_0^t (1+t-t')^{\max (1-\kappa,-1)}   d t'd t
        \\
        & \leq  g_\kappa(T) \norm{a}^2_{L^2(X)},
    \end{aligned}
\end{equation}
where
\begin{equation*}
    g_\kappa(T) = \begin{cases}
        \frac{\kappa 2^{4-\kappa} C_1}{(\kappa -1)(2-\kappa)} T^{3-\kappa} & \text{if $1<\kappa<2$,}\\
        4 C_1(T+1)\ln (T+1) & \text{if $\kappa\geq 2$.}
    \end{cases}        
\end{equation*}
To finalise the proof, we combine the estimates and identities in \eqref{EQ:free_HS_est_exponential_mixing_1}, \eqref{EQ:free_HS_est_exponential_mixing_3}, \eqref{EQ:free_HS_est_exponential_mixing_5}, and \eqref{EQ:free_HS_est_exponential_mixing_7} to obtain the stated estimate. This concludes the proof.    
\end{proof}

\subsection{Quantitive bound on the full Hilbert-Schmidt norm}

We will here combine the previously obtained estimates into a single estimate for the Hilbert-Schmidt norm of main interest. This is the content of the following lemma. 

\begin{lemma}\label{LE:Main_HS_est}
Let $X=\Gamma\setminus \bH$, where $\Gamma$ be a strictly hyperbolic Fuchsian group, and let $\tau\in\R$. 
For a $V\in L^\infty(X)$ consider the Schr\"odinger operator $H_V = -\Delta_X + V$ and for $t\geq 0$ the associated wave and modified wave propagator $P_V(t)$ and $\widetilde{P}_V(t,\tau)$. Moreover, suppose that $I\subset (\frac{1}{4},\infty)$ is a fixed compact interval and suppose that $a\in L^\infty(X)$ with zero average and that there exist $C_1>0$ and $\kappa>1$ such that  
\begin{equation*}
  \big|\langle a \circ \phi_r, a \rangle \big| \leq C_1 (t+1)^{-\kappa} \lVert a \rVert_{L^2(X)}^2,\quad r \geq 0.  
\end{equation*}
Then for all $0<r\leq \min(1, \frac{\InjRad_X}{2})$ and all $T\geq1$ we have
\begin{equation}
    \begin{aligned}
       \MoveEqLeft \Big\lVert \int_0^T \Pi_I(H_V) P_V(t) a P_V(t) \Pi_I(H_V) \, d t \Big\rVert_{\bHS}^2 
       \\
       \leq{}&2g_\kappa(T) \norm{a}^2_{L^2(X)}+  16\pi T^3 \norm{a}_{L^\infty(X)}^2 \vol(X(\leq 2T))
       \\
       &+ C_I T^7  \max(1, \lVert V\rVert_{L^\infty(X)}^4)  \lVert a \rVert^2_{L^\infty(X)}\left( \frac{  e^{4T+\frac12} \Vol(X(\leq2T))}{r^2} + \norm{V}_{L^2(X)}^2 \right)
        \end{aligned}
    \end{equation}
    and 
    \begin{equation}
    \begin{aligned}
       \MoveEqLeft \Big\lVert \int_0^T \Pi_I(H_V) \widetilde{P}_V(t,\tau) a P_V(t) \Pi_I(H_V) \, d t \Big\rVert_{\bHS}^2 
       \\
       \leq{}& 2g_\kappa(T) \norm{a}^2_{L^2(X)}+  16\pi T^3 \norm{a}_{L^\infty(X)}^2 \vol(X(\leq 2T))
       \\
       &+ C_I T^7  \max(1, \lVert V\rVert_{L^\infty(X)}^4)  \lVert a \rVert^2_{L^\infty(X)}\left( \frac{  e^{4T+\frac12} \Vol(X(\leq2T))}{r^2} + \norm{V}_{L^2(X)}^2 \right),
        \end{aligned}
    \end{equation}
where
\begin{equation*}
    g_\kappa(T) = \begin{cases}
        \frac{\kappa 2^{4-\kappa} C_1}{(\kappa -1)(2-\kappa)} T^{3-\kappa} & \text{if $1<\kappa<2$,}\\
        4 C_1(T+1)\ln (T+1) & \text{if $\kappa\geq 2$.}
    \end{cases}        
\end{equation*}
 and the constant is given by $C_{I}=\big(200 \pi \max(1,(\min(I)-\frac14)^{-1})\big)^2$ in both cases.
\end{lemma}

\begin{proof}
Again, the proofs of the two estimates are analogous, hence we will only prove the first estimate. We start by applying Lemma~\ref{LE:lifting_type_lemma} and obtain the bound
\begin{equation}\label{EQ:Main_HS_est_1}
    \begin{aligned}
       \MoveEqLeft \Big\lVert \int_0^T \Pi_I(H_V) P_V(t) a P_V(t) \Pi_I(H_V) \, d t \Big\rVert_{\bHS}^2 
       \\
       \leq{}& 2\Big\lVert \int_0^T \Pi_I(H_V) P_V(t) a \boldsymbol{1}_{X(>2T)} P_V(t) \Pi_I(H_V) \, d t  \Big\rVert_{\bHS}^2 
        \\
        &+ \frac{C_{I,1} T^7 e^{4T+1}\max(1, \lVert V\rVert_{L^\infty(X)}^4) \lVert a \rVert^2_{L^\infty(X)}}{r^2} \Vol(X(\leq2T)), 
        \end{aligned}
    \end{equation}
    where the constant $C_{I,1} = 16\pi \max(1,(\min(I)-\frac14)^{-1})$. Next,  we apply Lemma~\ref{LE:Potential_restriction_2} to the first term on the right-hand side of \eqref{EQ:Main_HS_est_1}. This gives us the estimate
    \begin{equation}\label{EQ:Main_HS_est_2}
    \begin{aligned}
        \MoveEqLeft \Big\lVert \int_0^T \Pi_I(H_V) P_V(t) a \boldsymbol{1}_{X(>2T)} P_V(t) \Pi_I(H_V) \, d t \Big\rVert_{\bHS}^2 
       \\
       &\leq 2\Big\lVert \int_0^T P_0(t) a \boldsymbol{1}_{X(>2T)} P_0(t)   \, d t \Big\rVert_{\bHS}^2 +  C_{I,2}  T^{7} \norm{a}_{L^\infty(X)}^2 \max(1,\norm{V}_{L^\infty(X)}^2) \norm{V}_{L^2(X)}^2, 
    \end{aligned}
\end{equation}
where the constant $C_{I,2}=\big(100 \sqrt{\pi} \max(1,(\min(I)-\frac14)^{-1})\big)^2$. Next,  we apply Lemma~\ref{LE:free_HS_est_exponential_mixing} to the first term on the right-hand side of \eqref{EQ:Main_HS_est_2}. This gives us the estimate

\begin{equation}\label{EQ:Main_HS_est_3}
       \Big\lVert \int_0^T  P_0(t) a \boldsymbol{1}_{X( > 2T)} P_0(t) \,d t  \Big\rVert_{\bHS}^2 \leq 2g_\kappa(T) \norm{a}^2_{L^2(X)}+  4\pi T^3 \norm{a}_{L^\infty(X)}^2 \vol(X(\leq 2T)).
\end{equation}
Combing the estimates from \eqref{EQ:Main_HS_est_1}, \eqref{EQ:Main_HS_est_2}, and \eqref{EQ:Main_HS_est_3} we can obtain the stated estimate, with constants as stated in the lemma. This concludes the proof. 
\end{proof}

\section{Proof of main result}\label{SEC:proof_main_thm}

In this section, we will give a proof of Theorem~\ref{thm:main}. Before doing so, let us recall the setup. We consider a sequence $ \{ X_{n} \}_{n\in\N} $ of compact connected hyperbolic surfaces that satisfies \textbf{(BSC)}, \textbf{(UND)}, and observables $a_n$ satisfying \textbf{(MIX)} and a sequence $\{ V_{n} \}_{n\in\N} $ of potentials which satisfies {\normalfont\textbf{(POT)}} with the two numbers $C_{\min}$ and $C_{\max}$. Then for each $n\in\N$ we let $H_V^{(n)}= -\Delta_{X_n}+V_n$ be a  Schr\"odinger operator with eigenvalues $\lambda_1^{(n)}\leq \lambda_2^{(n)}\leq \lambda_3^{(n)} \leq \lambda_4^{(n)} \cdots $ and eigenvectors $\{\psi_j\}_{j\in\N}$. Moreover, for all $n,j\in\N$ we set $\rho_{j}^{(n)} = \sqrt{\lambda_j^{(n)}-\frac{1}{4}}$ and $I_n=[a_n,b_n]$. Finally, we fix a closed interval $I=[a,b]\subset (\frac{1}{4},\infty)$ such that $b-C_{\max}>a  -\min(0, C_{\min})$ and $I'=[a',b']$ such that $I\subset I_n \subset I'$ for all $n\in\N$. We are now ready to prove the theorem.

\begin{proof}[Proof of Theorem~\ref{thm:main}]
We start by proving point (1). That is, we need to prove that for any uniformly bounded sequence of measurable functions $\{a_n\}_{n\in\N}$ we have
\begin{equation*}
        \lim_{n\rightarrow \infty}\frac{1}{\mathcal{N}\big(H_V^{(n)},I_n\big)} \sum_{j:\lambda_j^{(n)}\in I_n} \Big| \langle a_n \psi_j^{(n)} , \psi_j^{(n)}\rangle_{L^2(X)} - \frac{1}{\Vol(X_n)} \int_{X_n} a_n(x) \, dx \Big|^2 =0.
\end{equation*}
       
First, since the eigenfunctions are normalised, we can, without loss of generality, assume that $a_n$ has mean zero for all $n\in\N$. Letting $h:\R_+\times\R\mapsto\R$ be given by 
\begin{equation*}
    h(t,\lambda) = \big(\lambda - \frac{1}{4}\big)^{-\frac12}\sin\big(t\sqrt{\lambda - \tfrac{1}{4}}\big),
\end{equation*}
we get for any $\lambda_j^{(n)}\in I_n$ and $T>0$ that 
\begin{equation}
    \langle a_n \psi_j^{(n)} , \psi_j^{(n)}\rangle_{L^2(X_n)} = \frac{1}{\frac{1}{T} \int_0^T h^2(t,\lambda_j^{(n)}) \, d t}  \langle \frac{1}{T} \int_{0}^T P_V(t) a_n P_V(t) \, d t \, \psi_j^{(n)} , \psi_j^{(n)}\rangle_{L^2(X_n)},
\end{equation}
where we have used $h(t,H_V) = P_V(t)$. From Lemma~\ref{LE:lower_bound_h(t,lambda)} we have the lower bound
\begin{equation*}
    \begin{aligned}
    \inf_{\lambda \in I_n} \frac{1}{T} \int_0^T h^2(t,\lambda_j^{(n)}) \, d t \geq \inf_{\lambda \in I'} \frac{1}{T} \int_0^T h^2(t,\lambda_j^{(n)}) \, d t \geq \frac{1}{3(b'-\frac14)},
    \end{aligned}
\end{equation*}
for all $T$ sufficiently large (depending on $a'$ and $b'$) and all $n\in\N$. We obtain using this lower bound for all $n\in\N$ that 
\begin{equation}
    \Big|\langle a_n \psi_j^{(n)} , \psi_j^{(n)}\rangle_{L^2(X_n)}\Big| \leq  3b'  \Big|\langle \frac{1}{T} \int_{0}^T P_V(t) a_n P_V(t) \, d t \, \psi_j^{(n)} , \psi_j^{(n)}\rangle_{L^2(X_n)}\Big|.
\end{equation}
Using this, we obtain the following upper bound
\begin{equation}
    \begin{aligned}
    \MoveEqLeft \sum_{j:\lambda_j^{(n)}\in I_n} \Big| \langle a_n \psi_j^{(n)} , \psi_j^{(n)}\rangle_{L^2(X_n)}  \Big|^2 
    \\
    \leq{}& 9(b')^2 \sum_{j:\lambda_j^{(n)}\in I_n} \Big| \langle \frac{1}{T} \int_{0}^T P_V(t) a_n P_V(t) \, d t \psi_j^{(n)} , \psi_j^{(n)}\rangle_{L^2(X_n)}  \Big|^2
    \\
    ={}&  \frac{9b'^2}{T^2}  \sum_{j:\lambda_j^{(n)}\in I_n} \Big| \langle  \int_{0}^T \Pi_{I_n}(H_V^{(n)}) P_V(t) a_n P_V(t) \Pi_{I_n}(H_V^{(n)}) \, d t \psi_j^{(n)} , \psi_j^{(n)}\rangle_{L^2(X_n)}  \Big|^2
    \\
    \leq {}& \frac{9b'^2}{T^2} \Big\lVert \int_{0}^T \Pi_{I_n}(H_V^{(n)}) P_V(t) a_n P_V(t) \Pi_{I_n}(H_V^{(n)}) \, d t \Big\rVert_{\bHS}^2,
    \end{aligned}
\end{equation}
where we can freely insert the spectral projections since $\lambda_j^{(n)}\in I_n$ for all $j$. Applying Lemma~\ref{LE:Main_HS_est} to the Hilbert-Schmidt norm we obtain the following bound 
\begin{equation}
    \begin{aligned}
       \MoveEqLeft \Big\lVert \int_{0}^T \Pi_{I_n}(H_V^{(n)}) P_V(t) a_n P_V(t) \Pi_{I_n}(H_V^{(n)}) \, d t \Big\rVert_{\bHS}^2 
       \\
       \leq{}& 2 g_\kappa(T) \norm{a_n}^2_{L^2(X_n)}+  16\pi T^3 \norm{a_n}_{L^\infty(X)}^2 \vol(X(\leq 2T))
       \\
       &+ \tilde{C}_I T^7  \max(1, \lVert V\rVert_{L^\infty(X)}^4)  \lVert a_n \rVert^2_{L^\infty(X)}\left( \frac{  e^{4T+\frac12} \Vol(X(\leq2T))}{r^2} + \norm{V}_{L^2(X)}^2 \right),
        \end{aligned}
    \end{equation}
where $g_\kappa(T)$ is given by
\begin{equation*}
    g_\kappa(T) = \begin{cases}
        \frac{\kappa 2^{4-\kappa} C_1}{(\kappa -1)(2-\kappa)} T^{3-\kappa} & \text{if $1<\kappa<2$,}\\
        4 C_1(T+1)\ln (T+1) & \text{if $\kappa\geq 2$.}
    \end{cases}        
\end{equation*}
We may assume that $g_{\kappa}(T)$ is independent of $n$ due to assumption \textbf{(MIX)}. The number $r$ is given by $r=\min(1,\min_{n\in\N} (\frac{\InjRad_{X_n}}{2}))$. That $r>0$ and is bounded away from zero uniformly in $n$ follow from assumption \textbf{(UND)}. The constant $\tilde{C}_{I}$ does not depend on $n$ since we have $I\subset I_n$ for all $n\in\N$. In addition, the inclusion $I\subset I_n$ also gives us that 
\begin{equation}
    \mathcal{N}\big(H_V^{(n)},I_n\big)  \geq \mathcal{N}\big(H_V^{(n)},I\big) \geq c\Vol(X_n)
\end{equation}
for all $n\in\N$, where the last inequality follows from Lemma~\ref{LE:Weyl_schrodinger}. It follows from the above estimates, our assumption on the sequence of potentials, assumption \textbf{(BSC)}, and that $\norm{a_n}_{L^2(X_n)}^2 \leq \norm{a_n}_{L^\infty(X_n)}^2\Vol(X_n) $ that
    \begin{equation}
    \limsup_{n\rightarrow \infty}\frac{1}{\mathcal{N}\big(H_V^{(n)},I_n\big)} \sum_{j:\lambda_j^{(n)}\in I_n} \Big| \langle a_n \psi_j^{(n)} , \psi_j^{(n)}\rangle_{L^2(X_n)}  \Big|^2 \leq  \frac{2 g_\kappa(T)}{cT^2} \norm{a}^2_{L^\infty(X)}.
    \end{equation}
Then finally, taking $T$ to infinity, the desired result follows by observing that $g_\kappa(T) = o(T^2)$. This concludes the proof for point (1).

We will only prove point (3) as it also covers point (2) with $\tau=0$. To prove point (3) we need to establish that for any uniformly bounded sequence of measurable functions $\{a_n\}_{n\in\N}$ it holds that for every $\varepsilon>0$ there exists $\delta(\varepsilon)>0$ such that for all $\tau\in\R$
\begin{equation*}
        \limsup_{n\rightarrow \infty}\frac{1}{\mathcal{N}\big(H_V^{(n)},I_n\big)}  \sum_{\substack{j\neq k: \lambda_j^{(n)},\lambda_k^{(n)}\in I_n  \\ |\rho_{j}^{(n)}-\rho_{k}^{(n)}-\tau| < \delta(\varepsilon)}} \Big| \langle a_n \psi_j^{(n)} , \psi_k^{(n)}\rangle_{L^2(X_n)}  \Big|^2 <\varepsilon.
\end{equation*}  

Let $\tau\in\R$ and $\varepsilon>0$ be given. Since $\{\psi_j^{(n)}\}_{j\in\N}$ is an orthonormal basis for all $n$, we can, without loss of generality, assume that $a_n$ has mean zero. Moreover, we fix a number $m\in(\frac14,a)$ and set $\widetilde{m} = \sqrt{m-\frac14}$. Then for $\delta \in (0,\frac29\widetilde{m})$ we let $T_\delta = \frac{\pi}{2\delta}$, at the end of the proof we fix $\delta$ depending on $\varepsilon$.  
Letting $h,\tilde{h}_\tau:\R_+\times\R\mapsto\R$ be given by 
\begin{equation*}
    h(t,\lambda) = \big(\lambda - \frac{1}{4}\big)^{-\frac12}\sin\big(t\sqrt{\lambda - \tfrac{1}{4}}\big) \qquad\text{and}\qquad \tilde{h}_\tau(t,\lambda) = \cos(\tau t) \big(\lambda - \frac{1}{4}\big)^{-\frac12}\sin\big(t\sqrt{\lambda - \tfrac{1}{4}}\big),
\end{equation*}    
We get for $\lambda_j^{(n)},\lambda_k^{(n)} \in I_n$ such that $|\rho_j^{(n)}-\rho_k^{(n)}-\tau|<\delta$
\begin{equation*}
      \langle a_n \psi_j^{(n)} , \psi_k^{(n)}\rangle_{L^2(X_n)} = \frac{1}{\frac{1}{T_\delta} \int_0^{T_\delta} \tilde{h}_\tau(t,\lambda_{k}^{(n)}) h(t,\lambda_{j}^{(n)}) \, d t }  \langle \frac{1}{T_\delta} \int_{0}^{T_\delta} \widetilde{P}_V(t,\tau) a_n P_V(t) \, d t \, \psi_j^{(n)} , \psi_k^{(n)}\rangle_{L^2(X_n)},  
\end{equation*}
where we have used $h(t,H_V) = P_V(t)$ and $ h_\tau(t,H_V) = \widetilde{P}_V(t,\tau)$. From Lemma~\ref{LE:lower_bound_tilde_h_and_h(t,lambda)} we have
\begin{equation*}
    \inf_{\lambda_j^{(n)},\lambda_k^{(n)} \in I'} \frac{1}{T_\delta} \Big| \int_0^{T_\delta} \tilde{h}_\tau(t,\lambda_{k}^{(n)}) h(t,\lambda_{j}^{(n)}) \, d t \Big|^{-1} < 8\pi\big(b'-\tfrac{1}{4}\big).
\end{equation*}
Since $I_n\subset I'$ for all $n\in\N$ we obtain from the above estimate for all $\lambda_j^{(n)},\lambda_k^{(n)} \in I_n$ such that $|\rho_j^{(n)}-\rho_k^{(n)}-\tau|<\delta$ the estimate
\begin{equation*}
      \big|\langle a_n \psi_j^{(n)} , \psi_k^{(n)}\rangle_{L^2(X_n)}\big|^2 \leq \frac{64\pi^2 b'^2}{T_\delta^2} \Big|  \langle \int_{0}^{T_\delta} \widetilde{P}_V(t,\tau) a_n P_V(t) \, d t \, \psi_j^{(n)} , \psi_k^{(n)}\rangle_{L^2(X_n)} \Big|^2.  
\end{equation*}
Using this estimate, we obtain with an analogous estimate to the one used in the proof of point (1) the following estimate
\begin{equation*}
     \sum_{\substack{j\neq k: \lambda_j^{(n)},\lambda_k^{(n)}\in I_n \\ |\rho_{j}^{(n)}-\rho_{k}^{(n)}-\tau| < \delta(\varepsilon)}} \Big| \langle a_n \psi_j^{(n)} , \psi_k^{(n)}\rangle_{L^2(X_n)}  \Big|^2 
    \lesssim \frac{b'^2}{T_\delta^2} \Big\lVert \int_{0}^{T_\delta} \Pi_{I_n}(H_V^{(n)})\widetilde{P}_V(t,\tau) a_n P_V(t) \Pi_{I_n}(H_V^{(n)}) \, d t \Big\rVert_{\bHS}^2.
\end{equation*}
Applying Lemma~\ref{LE:Main_HS_est} to the Hilbert-Schmidt norm on the right-hand side, we obtain the estimate  
\begin{equation*}
    \begin{aligned}
     \MoveEqLeft \frac{b'}{T_\delta^2} \Big\lVert \int_{0}^{T_\delta} \Pi_{I_n}(H_V^{(n)})\widetilde{P}_V(t,\tau) a_n P_V(t) \Pi_{I_n}(H_V^{(n)}) \, d t \Big\rVert_{\bHS}^2 
    \\
    \leq{}& \frac{C_1 g_\kappa(T_\delta)}{T_\delta^2}  \norm{a_n}^2_{L^2(X)}+  16\pi T_\delta^3 \norm{a_n}_{L^\infty(X)}^2 \vol(X(\leq 2T_\delta))
       \\
       &+ C_2 T_\delta^5  \max(1, \lVert V\rVert_{L^\infty(X)}^4)  \lVert a_n \rVert^2_{L^\infty(X)}\left( \frac{  e^{4T_\delta +\frac12} \Vol(X(\leq2T_\delta))}{r^2} + \norm{V}_{L^2(X)}^2 \right),
    \end{aligned}
\end{equation*}
where $C_1$ and $C_2$ depend only on $I, I'$. Now, again arguing as in the proof of point (1) this gives us the following
\begin{equation*}
    \limsup_{n\rightarrow \infty}\frac{1}{\mathcal{N}\big(H_V^{(n)},I_n\big)}  \sum_{\substack{j\neq k: \lambda_j^{(n)},\lambda_k^{(n)}\in I_n \\ |\rho_{j}^{(n)}-\rho_{k}^{(n)}-\tau| < \delta(\varepsilon)}} \Big| \langle a_n \psi_j^{(n)} , \psi_k^{(n)}\rangle_{L^2(X_n)}  \Big|^2 \leq \frac{C_1 g_\kappa(T_\delta)}{T_\delta^2} \norm{a}^2_{L^\infty(X)},
\end{equation*}
where the constant only depends on $I$ and $I'$ and $g_\kappa$ is the same function as above. Recalling that $T_\delta= \frac{\pi}{2\delta}$ and $g_\kappa(T_\delta)=o(T_\delta^2)$, we see that by choosing $\delta$ sufficiently small, we obtain the desired estimate. This concludes the proof.
\end{proof}

\section{Proof of probabilistic main result}\label{SEC:outline_probabilistic_thm}

We do this just for a single fixed window $I = [a,b]$, the same argument holds for a sequence of windows $I_g$ as in the proof of Theorem \ref{thm:main}.  In order to establish Theorem~\ref{Thm:propalistic_main}, we will need the probabilistic versions of \textbf{(BSC)}, \textbf{(UND)}, and \textbf{(MIX)}. These are collected in the following Theorem.

\begin{theorem}\label{thm:properties_surface_large_genus}
    Suppose that $X$ is a random hyperbolic surface with the distribution of $X$ given by the Weil-Petersson distribution on the moduli space $\mathcal{M}_g$. Denote by
    \begin{equation*}
        0=\lambda_0<\lambda_1\leq \lambda_2 \leq \cdots 
    \end{equation*}
    the eigenvalues of the Laplace-Beltrami operator $-\Delta_X$. Then with probability $1-o(1)$ for large genus the following properties hold:
    \begin{enumerate}
        \item 
        \begin{equation*}
            \InjRad_X \geq g^{-\frac{1}{24}} \log(g)^{\frac{9}{16}}.    
        \end{equation*}

        \item 
        \begin{equation*}
            \frac{\Vol(\{ x\in X\,|\, \InjRad_X(x) <\frac{1}{6}\log(g)\}) }{\Vol(X)}=\mathcal{O}\big(g^{-\frac{1}{4}}\big).    
        \end{equation*}

        \item For any $\eta\in(0,\frac{1}{4})$
        \begin{equation*}
            \lambda_1 \geq \frac{1}{4} - \eta.
        \end{equation*}

        \item For any compact interval $I\subset (\frac{1}{4},\infty)$
        \begin{equation*}
            \Bigg|\frac{\mathcal{N}(-\Delta_X,I)}{\Vol(X)} - \frac{1}{4\pi} \int_\R \boldsymbol{1}_I (\tfrac{1}{4} +r^2) \tanh(\pi r) r \, dr \Bigg|=  \mathcal{O} \bigg(\sqrt{\max(I) -\frac{1}{4}} - \sqrt{\min(I)-\frac{1}{4}} +\sqrt{\tfrac{\max(I) +1}{\log(g)}} \bigg).
        \end{equation*}
    \end{enumerate}
\end{theorem}

Properties (1) and (2) in the above theorem are taken from \cite{monkBenjaminiSchrammConvergence2022}, property (4) is obtained similar to \cite[Theorem 6.2]{lemassonQuantumErgodicityEisenstein2024} using the results in \cite{monkBenjaminiSchrammConvergence2022}, and property (3) is taken from \cite{AnantharamanMonk2025,HideMaceraThomas2025b}, although up to changing the constants, a weaker result would suffice for our purposes (e.g. \cite{mirzakhaniGrowthWeilPeterssonVolumes2013} or \cite{WuXue}). The first three properties are those that replace assumptions \textbf{(BSC)}, \textbf{(UND)}, and \textbf{(MIX)} in this setting using the work of Ratner \cite{Ratner}. Property (4) we will use to establish a lower bound on the number of eigenvalues in an interval $I$ in this setting. This is the content of the following lemma. 

\begin{lemma}\label{LE:Weyl_schrodinger_proba}
Let $X$ be a random hyperbolic surface with the distribution of $X$ given by the Weil-Petersson distribution in the moduli space $\mathcal{M}_g$. Suppose that $V\in L^\infty(X)$ and consider the Schr\"odinger operator $H_V = -\Delta_X+V$. Then with probability $1-o(1)$ for large genus For any compact interval $I = [a,b]\subset (\frac{1}{4},\infty)$, such that $b-\max(V)>a  -\min(0, \min(V))$, there exist a positive constant $c$ and a natural number $N\in\N$ such that 
\begin{equation}
    \mathcal{N}(H_{V_n},I) \geq \Bigg( c - \mathcal{O} \bigg(\sqrt{\max(\tilde{I}) -\frac{1}{4}} - \sqrt{\min(\tilde{I})-\frac{1}{4}} +\sqrt{\tfrac{\max(\tilde{I}) +1}{\log(g)}} \bigg) \Bigg) \Vol(X). 
\end{equation}
where $\tilde{I}=[a-\min(0,\min(V),b- \max(V)]$.
\end{lemma}
\begin{proof}
 By arguing as in the proof Lemma~\ref{LE:Weyl_schrodinger} we obtain the following inequality
 \begin{equation*}
    \mathcal{N}(H_{V_n},I) \geq \mathcal{N}(-\Delta_X,\tilde{I}),
\end{equation*}
where $\tilde{I} = [a-\min(0,\min(V),b- \max(V)]$. Then using Theorem~\ref{thm:properties_surface_large_genus} point $(4)$, we obtain the stated estimate, which concludes the proof. 
\end{proof}

With this established, we can now give an outline for the proof of Theorem~\ref{Thm:propalistic_main} that provides details on how to modify the arguments used in \cite{lemassonQuantumErgodicityEisenstein2024} to include potentials.  Moreover, we will only give this outline of property (1) as the proof of properties (2) and (3) is obtained by analogous arguments.

\begin{proof}[Proof of Theorem~\ref{Thm:propalistic_main}]
We just need to find a sequence $\eps_g \to 0$ such that the probability of the event that a surface $X \in \cM_g$ satisfies for any $V$ and orthognormal basis $\psi_j^{(g)}$ on $X$:
\begin{equation*}
\frac{1}{\mathcal{N}\big(H_{V_g}^{(g)},I_g\big)} \sum_{j:\lambda_j^{(g)}\in I_g} \Big| \langle a_g \psi_j^{(g)} , \psi_j^{(g)}\rangle_{L^2(X_g)} - \frac{1}{\Vol(X_g)} \int_{X_g} a_g(x)  \, dx \Big|^2 \leq \eps_g,
\end{equation*}
converges to $1$ as $g \to \infty$. We may assume without the loss of generality that the average of $a_g$ is zero for all $g\in\N$. By arguing as in the proof of Theorem~\ref{thm:main} part (1) we obtain the following estimate
\begin{equation*}
    \begin{aligned}
        \MoveEqLeft \frac{1}{\mathcal{N}\big(H_{V_g}^{(g)},I_g\big)} \sum_{j:\lambda_j^{(g)}\in I_g} \Big| \langle a_g \psi_j^{(g)} , \psi_j^{(g)}\rangle_{L^2(X_g)}\Big|^2 
        \\
        \leq{}& \frac{1}{\mathcal{N}\big(H_{V_g}^{(g)},I_g\big)} \Bigg(  \frac{C}{\beta_1^3 T} \norm{a_g}^2_{L^2(X_g)}+  16\pi T \norm{a_g}_{L^\infty(X)}^2 \vol(X_g(\leq 2T))
       \\
       &+ \tilde{C}_I T^5  \max(1, \lVert V\rVert_{L^\infty(X_g)}^4)  \lVert a_g \rVert^2_{L^\infty(X_g)}\left( \frac{  e^{4T+\frac12} \Vol(X_g(\leq2T))}{r^2} + \norm{V}_{L^2(X_g)}^2 \right)\Bigg)
    \end{aligned}
\end{equation*}
where $r\leq \min(1,\frac{\InjRad_{X_g}}{2})$ and $\beta_1 = 1-\sqrt{1-4\lambda_1^{(g)}}$. Moreover, instead of applying Lemma~\ref{LE:int_est_after_exp_mix} we use \cite[Lemma~{6.6}]{Hippi} to obtain the factor $\frac{C}{\beta_1^3 T}$ instead of $g_\kappa(T)$ as in the previous proof. By applying Theorem~\ref{thm:properties_surface_large_genus} and Lemma~\ref{LE:Weyl_schrodinger_proba} we obtain with high probability and for large genus the upper bound 
\begin{equation*}
    \begin{aligned}
         \frac{1}{\mathcal{N}\big(H_{V_g}^{(g)},I_g\big)} \sum_{j:\lambda_j^{(g)}\in I} \Big| \langle a_g \psi_j^{(g)} , \psi_j^{(g)}\rangle_{L^2(X_g)}\Big|^2 
        \leq{}& C_g \Bigg[  \frac{C_1\norm{a_g}^2_{L^\infty(X_g)}}{(1-2\sqrt{\eta})^3 T } +  \frac{C_2 T\norm{a_g}_{L^\infty(X)}^2 \vol(X_g(\leq 2T))}{\Vol(X_g)}
       \\
       & + \tilde{C}_{I} \left( \frac{ T^5  e^{4T+\frac12} g^{\frac{1}{12}} \Vol(X_g(\leq2T))}{\log(g)^{\frac{9}{8}} \Vol(X_g) } + \frac{T^5\norm{V}_{L^2(X_g)}^2}{\Vol(X_g)} \right)\Bigg],
    \end{aligned}
\end{equation*}
where the constants will contain the missing terms from the previous estimate with the exception of $C_g$. The constant $C_g$ will satisfy that 
\begin{equation*}
    C_g^{-1} =  c- \mathcal{O}\bigg(\sqrt{\max(\tilde{I}) -\frac{1}{4}}\sqrt{\min(\tilde{I})-\frac{1}{4}} +\sqrt{\tfrac{\max(\tilde{I}) +1}{\log(g)}} \bigg)
\end{equation*}
where $\tilde{I} = [a-\min(0,\min(V),b- \max(V)]$ and $c$ is some positive constant. Now by choosing 
\begin{equation*}
    T = \min\Bigg( \bigg(\frac{\norm{V}_{L^2(X_g)}^2}{\Vol(X_g)}\bigg)^{-1/10},  c\log(g) \Bigg)
\end{equation*} 
for some small $c \in (0, 1/6)$ and then taking genus to infinity we see that we obtain the desired results with high probability for a suitable $\eps_g \to 0$ due to the properties (1)-(4) of Theorem \ref{thm:properties_surface_large_genus}, see \cite{lemassonQuantumErgodicityEisenstein2024} for a more quantitative deduction in the zero potential case.  As mentioned before the outline of proof, the proof of part (2) and (3) are analogous to the one given above.
\end{proof}

\bibliographystyle{plain}
\bibliography{Mathematics}

\end{document}